\pdfoutput=1
\documentclass[11pt,a4paper,reqno]{amsart}

\usepackage[margin=1.0in]{geometry}
\usepackage{graphicx}
\usepackage{ amssymb }
\usepackage{ textcomp }
\usepackage[utf8]{inputenc}
\usepackage{amsfonts}
\newtheorem{theorem}{Theorem}[section]
\newtheorem{lemma}[theorem]{Lemma}
\newtheorem{proposition}[theorem]{Proposition}

\usepackage{upgreek}
\theoremstyle{definition}

\usepackage{url}
\usepackage{hyperref}

\def\R{\mathbb{R}}

\def\C{\mathbb{C}}
\def\Z{\mathbb{Z}}

\usepackage{amsthm}

\begin{document}
	
	\newtheorem{thm}{Theorem}[section]
	\numberwithin{thm}{section}
	\newtheorem{lem}[thm]{Lemma}
	\newtheorem{propo}[thm]{Proposition}
	\newtheorem{corol}[thm]{Corollary}
	\newtheorem{deff}[thm]{Definition}
	\numberwithin{equation}{section}
	
	\title{Decouplings for three-dimensional surfaces in $\R^{6}$
	}
	
	
	\author{Changkeun Oh 
	}
	\address{Department of Mathematics, Pohang University of Science and Technology, Pohang 790-784, Republic of Korea}
\email{ock9082@postech.ac.kr}

\maketitle
\begin{abstract}
		We obtain the sharp $l^p$ decoupling for three-dimensional nondegenerate surfaces in $\mathbb{R}^6$. This can be thought of as a generalization of Bourgain and Demeter's result, which is the sharp $l^p$ decoupling for two-dimensional nondegenerate surfaces in $\R^4$.

	\end{abstract}
	
\section{Introduction}
\label{sec:1}

 Consider a three-dimensional surface in $\R^{6}$
\begin{displaymath}
\begin{split}
S=\{(\xi_{1},\xi_{2},\xi_{3},\Phi_{1}(\xi_{1},\xi_{2},\xi_{3}),\Phi_{2}(\xi_{1},\xi_{2},\xi_{3}),\Phi_{3}(\xi_{1},\xi_{2},\xi_{3})): (\xi_{1},\xi_{2},\xi_{3}) \in [0,1]^3	\}.
\end{split}
\end{displaymath} 
We assume that the functions $\Phi_{1},\Phi_{2},\Phi_{3}$ are homogeneous polynomials of degree two. In addition to this, we assume that the Jacobian of $(\Phi_{1},\Phi_{2},\Phi_{3}): \R^3 \rightarrow \R^3$ is not identically zero:
\begin{displaymath}
\LARGE\begin{vmatrix}
\frac{\partial{\Phi_{1}}}{\partial \xi_{1}} & \frac{\partial{\Phi_{1}}}{\partial \xi_{2}} & \frac{\partial{\Phi_{1}}}{\partial \xi_{3}} \\[0.30em]
\frac{\partial{\Phi_{2}}}{\partial \xi_{1}} & \frac{\partial{\Phi_{2}}}{\partial \xi_{2}} & \frac{\partial{\Phi_{2}}}{\partial \xi_{3}} \\[0.30em]
\frac{\partial{\Phi_{3}}}{\partial \xi_{1}} & \frac{\partial{\Phi_{3}}}{\partial \xi_{2}} & \frac{\partial{\Phi_{3}}}{\partial \xi_{3}} 
\end{vmatrix} \not\equiv 0.
\end{displaymath} 
We call this class of surfaces  three-dimensional \textit{nondegenerate} surfaces in $\R^6$.
For simplicity, we use the notation $\xi=(\xi_{1},\xi_{2},\xi_{3}) \in \R^3$ and $e(x)=e^{2\pi ix}$ for $x \in \mathbb{R}$, and we define a function $\Phi: [0,1]^3 \rightarrow \R^3$ by $\Phi(\xi)=(\Phi_{1}(\xi),\Phi_{2}(\xi),\Phi_{3}(\xi))$.

Given a function $g:[0,1]^3 \rightarrow \C$ and a rectangular box $\theta \subset [0,1]^3$, we define the $\textit{extension}$ operator $E_{\theta}$ associated with the surface $S$ by
$$
E_{\theta}g(x)=\int_{\theta} g(\xi)e(x_{1}\xi_{1}+x_{2}\xi_{2}+x_{3}\xi_{3}+(x_{4},x_{5},x_{6}) \cdot \Phi(\xi)) \, d\xi_{1} d\xi_{2} d \xi_{3}.
$$
For a positive weight $w:\R^6 \rightarrow (0,\infty)$, define the weighted $L^p$ norm
$$
\|f\|_{L^p(w)} = (\int_{\R^6}|f(x)|^pw(x)\,dx )^{\frac{1}{p}}.
$$
For a ball $B_{N}$ centered at $c(B)$ with radius $N$, we let $w_{B}$ denote the weight
$$
w_{B}(x) = \frac{1}{(1+\frac{|x-c(B)|}{N})^C},
$$
where the constant $C$ is a large but unspecific constant. For $N \geq 1$, let $\mathcal{P}_{N^{-1}}$ be the collection of all cubes $c+[0,\frac{1}{N^{1/2}}]^3$ with $c \in \frac{1}{N^{1/2}}\Z^{3}$. 

Given $N \geq 1$, $p \geq 2$ and a nondegenerate surface $S$, let $D_{S}(N,p)$ be the smallest constant satisfying the following inequality
$$  
\| E_{[0,1]^3}f \|_{L^p(w_{B_{N}})} 
\leq D_{S}(N,p) \bigl(\sum\limits_{\substack{\theta \in \mathcal{P}_{N^{-1}}}} \| E_{\theta}f \|_{L^p(w_{B_{N}})}^p \bigr)^{\frac{1}{p}}
$$
for any $f:[0,1]^3 \rightarrow \C$.

The $l^p$ decoupling problem for $S$ is to find a sharp upper bound of $D_{S}(N,p)$. Our main result is as follows.

\begin{theorem}
	Let $S$ be a three-dimensional nondegenerate surface in $\R^6$.
	\begin{enumerate}
		\item If there is some two-dimensional plane $L$ in $\R^3$ satisfying $\Phi(L) =0$, then for every $p \geq 2$ and $\epsilon>0$, there exists $C_{p,\epsilon}$ such that for every $N \geq 1$
		$$
		D_{S}(N,p) \leq C_{p,\epsilon} N^{\epsilon} \max{(N^{\frac{3}{2}-\frac{6}{p}},N^{1-\frac{2}{p}})}.
		$$
		\item If such $L$ does not exist, then for every $p \geq 2$ and $\epsilon>0$, there exists $C_{p,\epsilon}$ such that for every $N \geq 1$ 
		$$
		D_{S}(N,p) \leq C_{p,\epsilon} N^{\epsilon} \max{(N^{\frac{3}{2}-\frac{6}{p}},N^{\frac{3}{2}(\frac{1}{2}-\frac{1}{p})})}.
		$$
	\end{enumerate}
\end{theorem}

Theorem 1.1 is sharp up to $N^\epsilon$ losses. For the sharpness, we refer to the discussion in Section 9 of the paper \cite{DGS-Sharp-2017}. For a $d$-dimensional surface in $\R^{2d}$, one can follow the same arguments used in the proof of Theorem 1.1, but this does not give a sharp decoupling because of a lack of lower dimensional decouplings.
\\

The decoupling problem was introduced by Wolff \cite{W-Local-2000} in connection with the Bochner-Riesz type problem. After that, some progress in the $l^p$ decoupling for hypersurfaces has been made in \cite{LW-A-2002,LP-Wolff-2006, GSS-Improvements-2008,GS-On-2009,GS-A-2010, B-Moment-2013,D-Incidence-2014}.  
Recently, Bourgain and Demeter \cite{BD-The-2015} proved the decoupling conjectures for the paraboloid and the cone. Based on their results, the $l^p$ decouplings for surfaces of codimension larger than one have been studied in \cite{BD-Decouplings-2016,BD-Mean-2016,BDG-Proof-2016,BDG-Sharp-2017}. Moreover, it is known that their results are closely related to number theory, in particular, the estimates on the Riemann zeta function on the critical line and Parsell-Vinogradov systems. Our result can also be used to obtain the upper bound of the number of the system of Diophantine equations associated with the surface $S$; see \cite{BD-Mean-2016,BDG-Proof-2016}. 
\\

Our proof basically follows the Bourgain and Demeter's framework \cite{BD-The-2015}. First, we show that the sharp multilinear version of a decoupling implies the sharp linear decoupling; then we obtain the multilinear version of a decoupling. The main problem happens in the first step, in which we must deal with the decouplings clustered near 3-variety. Bourgain and Demeter solved this problem for other surfaces. Decouplings for the paraboloid require decouplings clustered near 1-variety instead of 3-variety. Because 1-variety is a hyperplane, Bourgain and Demeter could simply use a lower dimensional decoupling to solve the problem, and for the decouplings for two-dimensional surfaces, Bourgain and Demeter solved the problem just by applying trivial decouplings; neither approach applied directly to our problem. The novelty of our proof is to present a general way to apply lower dimensional decouplings to deal with the decouplings clustered near arbitrary variety.

My ideas are as follows. The first step is to show that we can essentially assume that the 3-variety is a zero set of ``non-singular'' polynomial, i.e., a manifold. This proof makes use of a decomposition of the 3-variety and an iteration with different scales. The next step is a variation of the iteration argument in \cite{PS-regularity-2007} (See also \cite{GO-Remarks-2016} and Section 8 in \cite{BD-The-2015}). We first approximated the manifold to tangent planes at some scale. Because a tangent plane is 1-variety, we can directly apply a lower dimensional decoupling as Bourgain and Demeter did for the paraboloid. Repeating this argument with rescaling completes the proof.

It is likely to have to handle a decoupling clustered near a variety to apply the framework of Bourgain and Demeter to surfaces of codimension larger than one.
Our argument for the decouplings clustered near arbitrary variety for a three-dimensional surface in $\R^6$ does not rely on any property of this surface. Hence, our argument can be used to obtain the decouplings clustered near arbitrary variety for other surfaces of codimension larger than one. 
\\

Throughout the paper we assume that the numbers $\epsilon^{-1}$ and $N^{2^{-m}}$ for some large integer $m$ are dyadic numbers to avoid technical problems. We write $A \lesssim B$ if $A \leq cB$ and $A \sim B$ if $c^{-1}A \leq B \leq cA$.  The constant $c$ will in general depend on fixed parameter $p$ and sometimes on the variable parameter $\epsilon$ but not $N$. 
If $R$ is a rectangular box and $c$ is a positive real number, then we denote by $cR$ the box obtained by dilating $R$ by a number $c$ about its center. Throughout the paper, our surface $S$ is always a three-dimensional surface in $\R^6$.

\subsection{Outline of the paper}
\label{sec:1.1}
In Section 2, we define a transversality and obtain the multilinear Kakeya inequality. In Section 3, we give some definitions and get some lemmas. 
In Section 4, we obtain a weak form of a decoupling clustered near a variety, which contains the novelty of this paper.
In Section 5, we study relations between the linear $l^p$ decoupling and the multilinear $l^p$ decoupling. 
In Section 6, we give well-known equivalent formulations of the decoupling problem. In Section 7, we review a standard wave packet decomposition. In Section 8, we complete the proof of Theorem 1.1. 

\section{Transversality}
\label{sec:2}

In this section, we will study a transversality condition. More precisely, we will define some concepts related to the transversality condition, and then we will obtain the multilinear Kakeya inequality.  Since our surface $S$ is a submanifold of half the ambient dimension, the multilinear Kakeya inequality is the same as the bilinear Kakeya inequality.

\subsection{Definitions}
\label{sec:2.1}
We can take three linearly independent normal vectors to $S$ at $(p,\Phi(p))$:
\begin{gather*}
	m_{1}(p)=(\frac{\partial \Phi_{1}}{\partial \xi_{1}}(p),\frac{\partial \Phi_{1}}{\partial \xi_{2}}(p),\frac{\partial \Phi_{1}}{\partial \xi_{3}}(p),-1,0,0),
	\\
	m_{2}(p)=(\frac{\partial \Phi_{2}}{\partial \xi_{1}}(p),\frac{\partial \Phi_{2}}{\partial \xi_{2}}(p),\frac{\partial \Phi_{2}}{\partial \xi_{3}}(p),0,-1,0),
	\\
	m_{3}(p)=(\frac{\partial \Phi_{3}}{\partial \xi_{1}}(p),\frac{\partial \Phi_{3}}{\partial \xi_{2}}(p),\frac{\partial \Phi_{3}}{\partial \xi_{3}}(p),0,0,-1).
\end{gather*}
Fix $\nu>0$. We say that two points $p_{1}$, $p_{2}$ in $\mathbb{R}^3$ are \textit{$\nu$-transverse} if
$$
J(p_{1},p_{2})=|\mathrm{det}(m_{1}(p_{1}),m_{2}(p_{1}),m_{3}(p_{1}),m_{1}(p_{2}),m_{2}(p_{2}),m_{3}(p_{2}))|>\nu.
$$
Note that
$
J(p_{1},p_{2})=J(p_{1}-p_{2},0)$.
This symmetry will make our proof easier. We say that two sets $E_{1},E_{2} \subset \mathbb{R}^3$ are $\nu$-transverse if any two points $p_{1} \in E_{1}$ and $p_{2} \in E_{2}$ are $\nu$-transverse.


\subsection{The bilinear Kakeya inequality}
\label{sec:2.2}

\begin{lem}[The bilinear restriction theorem]
	Let $\nu>0$. Let $R_{1},R_{2}$ be $\nu$-transverse cubes in $[0,1]^3$. Then for each $g_{i} : R_{i} \rightarrow \mathbb{C}$, we have
	$$\||E_{R_{1}}g_{1}E_{R_{2}}g_{2}|^{\frac{1}{2}}\|_{L^4(\mathbb{R}^{6})} \lesssim_{\nu} (\|g_{1}\|_{L^2(R_{1})}\|g_{2}\|_{L^2(R_{2})})^{\frac{1}{2}}.$$
\end{lem}

\begin{proof}
	We use the change of variables
	\begin{displaymath}
	\begin{split}
	(\xi_{1},\xi_{2},\xi_{3},\eta_{1},\eta_{2},\eta_{3})  \mapsto (\xi_{1}+\eta_{1},\xi_{2}+\eta_{2},\xi_{3}+\eta_{3},\Phi(\xi_{1},\xi_{2},\xi_{3})+\Phi(\eta_{1},\eta_{2},\eta_{3})).
	\end{split}
	\end{displaymath}
	Note that the Jacobian of this mapping is given by $J(\xi_{1},\xi_{2},\xi_{3},\eta_{1},\eta_{2},\eta_{3})$.
	Since this transformation is defined in terms of homogeneous polynomials and $J(\xi,\eta) \neq 0$ for $\xi \in R_{1}$ and $\eta \in R_{2}$, it follows from Bezout's theorem that it has a uniformly bounded multiplicity. Hence
	$$|E_{R_{1}}g_{1}(x)E_{R_{2}}g_{2}(x)|=\widehat{FJ^{-1}}(x),$$
	where $F=g_{1}g_{2}$ and $|J^{-1}(u)| < \nu^{-1}$.
	By using Plancherel's identity, we have
	$$\||E_{R_{1}}g_{1}E_{R_{2}}g_{2}|^{\frac{1}{2}}\|_{L^4}=\|FJ^{-1}\|_{L^2}^{\frac{1}{2}} \lesssim_{\nu}  (\int|F^2(u){J^{-1}(u)}| \,du)^{\frac{1}{4}}=(\|g_{1}\|_{2}\|g_{2}\|_{2})^{\frac{1}{2}}.$$
	This completes the proof of Lemma 2.1.
\end{proof}

Let $\mathcal{P}$ be a collection of all cubes $P_{p}$ on a three-dimensional affine subspaces in $\mathbb{R}^{6}$
with a point $p \in [0,1]^3$ satisfying the following: the side lengths of each $P_{p}$ are equal to $N^{\frac{1}{2}}$ and the axes of $P_{p}$ span a subspace spanned by three vectors $m_{1}(p), m_{2}(p) ,m_{3}(p)$. We will say that $P_{p}$ is \textit{associated with} $p$.

\begin{deff}
	Let $\nu>0$. We say that two families $\mathcal{P}_{i}$ are $\nu$-transverse if there exist two cubes $\alpha_{1}, \alpha_{2} \subset  \R^3$, which are $\nu$-transverse, such that each $P_{p} \in \mathcal{P}_{i}$ is associated with a point $p \in \alpha_{i}$. 
\end{deff}

Suppose $P_{j,a}$ are elements of $\mathcal{P}_{j}$, where $1\leq a \leq N_{j}$ and  $j=1,2$. We denote by $\tilde{P}_{j,a}$ the 1-neighborhood of $P_{j,a}$ in $\mathbb{R}^{6}$, and denote by $T_{j,a}$ the characteristic function of $\tilde{P}_{j,a}$.

\begin{propo}[The bilinear Kakeya-type inequality] Let $\nu>0$. Assume that two families $\mathcal{P}_{j}=\{P_{j,a}:1 \leq a \leq N_{j}	\}$, $j=1,2$, are $\nu$-transverse. Then we have
	$$\int_{\mathbb{R}^{6}}\sum_{a=1}^{N_{1}}\sum_{b=1}^{N_{2}}T_{1,a}(x)T_{2,b}(x)\,dx \lesssim_{\nu} N_{1}N_{2}.$$
\end{propo}
There is a well known argument proving that the restriction inequality implies the Kakeya \mbox{inequality}. To prove Proposition 2.3, we follow the same argument.
\begin{proof}
	Fix an element $P_{i,a} \in \mathcal{P}_{i}$ and a point $v_{i,a} \in \mathbb{R}^{6}$. Let $P_{i,a}$ be associated with some point $p_{i,a}$. Let $w_{i,a}=p_{i,a}+[0,N^{-\frac{1}{2}}]^3 \subset [0,1]^3$. Define $f_{w_{i,a}}: [0,1]^3 \rightarrow \mathbb{C}$ by $$f_{w_{i,a}}(\xi)=e(-v_{i,a} \cdot (\xi,\Phi(\xi)))\chi_{w_{i,a}}(\xi).$$ Then 
	\begin{displaymath}
	\begin{split}
	|E_{[0,1]^3}f_{w_{i,a}}(x)|=|\int_{w_{i,a}} e((x-v_{i,a}) \cdot ((\xi,\Phi(\xi))-(p_{i,a},\Phi(p_{i,a}))))\, d\xi|.
	\end{split}
	\end{displaymath}
	Let $\tilde{P}_{i,a}(v_{i,a})$ be the translation of the rectangle $cN^{\frac{1}{2}}\tilde{P}_{i,a}$ centered at $v_{i,a}$ for some small constant $c>0$. Then for any $x \in \tilde{P}_{i,a}(v_{i,a})$
	\begin{displaymath}
	\begin{split}
	|E_{[0,1]^3}f_{w_{i,a}}(x)| \geq \int_{w_{i,a}}\frac{1}{1000} d\xi  = \frac{1}{1000N^{\frac{3}{2}}}.
	\end{split}
	\end{displaymath}
	Hence, we get $|E_{[0,1]^3}f_{w_{i,a}}(x)| \geq {10^{-3}N^{-\frac{3}{2}}}\chi_{\tilde{P}_{i,a}(v_{i,a})}(x)$, so
	$$ 
	|E_{[0,1]^3}f_{w_{1,a}}(x)E_{[0,1]^3}f_{w_{2,b}}(x)|
	\geq  {10^{-6}}N^{-3}\chi_{\tilde{P}_{1,a}(v_{1,a})}(x)\chi_{\tilde{P}_{2,b}(v_{2,b})}(x)
	$$
	for any transverse sets $w_{1,a}$ and $w_{2,b}$, for any points $v_{1,a},v_{2,b} \in \mathbb{R}^{6}$ and for any $x \in \mathbb{R}^{6}$.
	By the bilinear restriction theorem, we get
	\begin{gather*}
		N^{-3} N_{1}N_{2}  \gtrsim \|(\sum_{a=1}^{N_{1}}|f_{w_{1,a}}|^2)^{\frac{1}{2}}\|_{L^2(\mathbb{R}^{3})}^2 \|(\sum_{b=1}^{N_{2}}|f_{w_{2,b}}|^2)^{\frac{1}{2}}\|_{L^2(\mathbb{R}^{3})}^2 
		\\ \gtrsim
		\|(\sum_{a,b}|E_{[0,1]^3}f_{w_{1,a}}(x)E_{[0,1]^3}f_{w_{2,b}}(x)|^2)^{\frac{1}{2}}\|_{L^2(\mathbb{R}^{6})}^2
		\gtrsim 
		N^{-6}\int_{\mathbb{R}^{6}} \sum_{a,b}\chi_{\tilde{P}_{1,a}(v_{1,a})}(x)\chi_{\tilde{P}_{2,b}(v_{2,b})}(x)\,dx.
	\end{gather*}
	Next, we use the change of variables $y=cN^{-\frac{1}{2}} x$. This gives the desired results.
\end{proof}

By interpolating two points $p=4$ and $p=\infty$ via H\"{o}lder's inequality, we obtain
\begin{displaymath}
\begin{split}
\int_{\mathbb{R}^{6}} |\sum_{a=1}^{N_{1}}\sum_{b=1}^{N_{2}}T_{1,a}(x)T_{2,b}(x)|^{\frac{p}{4}}dx \lesssim_{\nu} (N_{1}N_{2})^{\frac{p}{4}}
\end{split}
\end{displaymath}
for any $4 \leq p < \infty$. Moreover, a standard argument gives
\begin{corol} Let $\nu>0$.
	Assume that two families $\mathcal{P}_{j}=\{P_{j,a}:1 \leq a \leq N_{j}	\}$, $j=1,2$, are $\nu$-transverse. Then for any $4 \leq p < \infty$, we have
	\begin{displaymath}
	\begin{split}
	\int_{\mathbb{R}^{6}}\prod_{i=1}^{2} |\sum_{a=1}^{N_{i}}T_{i,a}*g_{i,a}(x)|^{\frac{p}{4}}dx \lesssim_{\nu} \prod_{i=1}^{2}(\sum_{a=1}^{N_{i}}\|g_{i,a}\|_{L^1(\mathbb{R}^{6})})^{\frac{p}{4}}
	\end{split}
	\end{displaymath}
	for all nonnegative functions $g_{i,a} \in L^{1}(\mathbb{R}^{6})$.
\end{corol}

\begin{proof}
	First, observe that 
	\begin{equation}
	\begin{split}
	\int_{\mathbb{R}^{6}} |\sum_{a=1}^{\infty}\sum_{b=1}^{\infty}u_{a}v_{b}T_{1,a}(x)T_{2,b}(x)|^{\frac{p}{4}}dx \lesssim_{\nu} (\sum_{a=1}^{\infty}\sum_{b=1}^{\infty}u_{a}v_{b})^{\frac{p}{4}}
	\end{split}
	\end{equation}
	for all non-negative real numbers $u_{a}$,$v_{b}$. Let $c_{i,a}$ be the center of $\tilde{P}_{i,a}$. Next, we take a finitely overlapping cover of $\mathbb{R}^{6}$ by translating a fixed tube $\tilde{P}_{i,a}$, and call this cover $\mathcal{G}$. Then
	\begin{displaymath}
	\begin{split}
	T_{i,a}*g_{i,a}(x) =\int_{x-\tilde{P}_{i,a}}g_{i,a}(y)\,dy \leq \sum_{P \in \mathcal{G}} ({\tilde{T}}_{i,a,P}(x) \int_{P}g_{i,a}(y)\,dy ),
	\end{split}
	\end{displaymath}
	where $\tilde{T}_{i,a,P}$ is a characteristic function of $100P+c_{i,a}$. Combining this with $(2.1)$ gives the desired results.
\end{proof}

\section{Some definitions and lemmas}
\label{sec:3}

Let $S$ be a three-dimensional nondegenerate surface in $\mathbb{R}^{6}$. Recall that for given $p \geq 2$ and $N \geq1$, $D_{S}(N,p)$ is defnied as the smallest constant satisfying
$$\|E_{S}g\|_{L^p({w}_{B_{N}})} \leq D_{S}(N,p)\bigl(\sum\limits_{\substack{\theta
		\in \mathcal{P}_{N^{-1}} }}\|E_{\theta}g\|_{L^p({w}_{B_{N}})}^p\bigr)^{\frac{1}{p}}$$
for each $g: [0,1]^{3} \rightarrow \mathbb{C}$. For $\nu>0, p \geq 2$ and $N\geq 1$, we denote by $D_{bil}(N,p,\nu)$ the smallest constant such that the bilinear decoupling holds; 
$$\||E_{R_{1}}g_{1}E_{R_{2}}g_{2}|^{\frac{1}{2}}\|_{L^p({w}_{B_{N}})} \leq D_{bil}(N,p,\nu)\biggl(\prod_{i=1}^2
\sum\limits_{\substack{\theta \in \mathcal{P}_{N^{-1}} }}
\|E_{\theta}g_{i}\|_{L^p({w}_{B_{N}})}^p\biggr)^{\frac{1}{2p}}$$
for any functions $g_{1},g_{2} : [0,1]^3 \rightarrow \mathbb{C}$ and any $\nu$-transverse dyadic cubes $R_{1},R_{2} \subset [0,1]^3$.

We will use the following lemma frequently. This lemma is identical to Lemma 7.1 in \cite{BDG-Proof-2016}.
\begin{lem}[A localization principle]
	Let $\mathcal{W}$ be the collection of positive integrable functions on $\mathbb{R}^{6}$. Let $O_{1},O_{2}:\mathcal{W} \rightarrow [0,\infty]$ have the following four properties:
	\begin{enumerate}
		\item $O_{1}(1_{B}) \lesssim O_{2}(w_{B})$ for all cubes $B \subset \mathbb{R}^{6}$ of side length $R$
		\item $O_{1}(u+v) \leq O_{1}(u)+O_{1}(v)$, for each $u,v \in \mathcal{W}$
		\item $O_{2}(u+v) \geq O_{2}(u)+O_{2}(v)$, for each $u,v \in \mathcal{W}$
		\item If $u \leq v$ then $O_{i}(u) \leq O_{i}(v)$.
	\end{enumerate}
	Then
	\begin{displaymath}
	\begin{split}
	O_{1}(w_{B}) \lesssim O_{2}(w_{B})
	\end{split}
	\end{displaymath}
	for each cube $B$ with side length $R$. The implicit constant is independent of $R,B$ and depends on the implicit constant from $(1)$.
\end{lem}


One of the key propositions is the parabolic rescaling. The proof of Proposition 3.2 is identical to that of Proposition 7.1 in \cite{BD-Mean-2016}.

\begin{propo}[Parabolic rescaling]
	Let $N,\sigma$ be numbers with $0< N^{-1} \leq \sigma $, and let $\tau=a+[0,\sigma^{\frac{1}{2}}]^3 \in \mathcal{P}_{\sigma}$. Then for each $f: [0,1]^3 \rightarrow \mathbb{C}$, we have
	\begin{displaymath}
	\begin{split}
	\|E_{\tau}f\|_{L^p({w}_{B_{N}})} \lesssim 
	{D}_{S}({N\sigma},p)
	\bigl( \sum\limits_{\substack{\theta \in \mathcal{P}_{N^{-1}} },\, \theta \subset \tau}\|E_{\theta}f\|_{L^p({w}_{B_{N}})}^p	\bigr)^{\frac{1}{p}}.
	\end{split}
	\end{displaymath}	
\end{propo}

\begin{proof}
	By lemma 3.1, it suffices to show that
	\begin{displaymath}
	\begin{split}
	\|E_{\tau}f\|_{L^p({B_{N}})} \lesssim 
	{D}_{S}({N\sigma},p)
	\bigl( \sum\limits_{\substack{\theta \in \mathcal{P}_{N^{-1}} },\, \theta \subset \tau}\|E_{\theta}f\|_{L^p({w}_{B_{N}})}^p	\bigr)^{\frac{1}{p}}.
	\end{split}
	\end{displaymath}
	We write $a=(a_{1},a_{2},a_{3})$ and define an affine transformation associated with $\tau$ by
	\begin{equation}
	\begin{split}
	&L_{\tau}(\xi_{1},\xi_{2},\xi_{3})
	=
	(\frac{\xi_{1}-a_{1}}{\sigma^{1/2}},\frac{\xi_{2}-a_{2}}{\sigma^{1/2}},\frac{\xi_{3}-a_{3}}{\sigma^{1/2}})
	\end{split}
	\end{equation}
	so that the image of $ \tau$ under $L_{\tau}$ is $ [0,1]^{3}$.
	Define $g(\xi)=f(L_{\tau}^{-1}\xi)\sigma^{\frac{3}{2}-\frac{9}{2p}}$. Through routine calculations, we can see 
	$$\sigma^{\frac{9}{2p}}|E_{[0,1]^3}g(\sigma^{\frac{1}{2}}((x_{1},x_{2},x_{3})+A_{\tau}(x_{4},x_{5},x_{6})),\sigma (x_{4},x_{5},x_{6})))|=|E_{\tau}f(x)|,$$
	where $A_{\tau}$ is some $3$ by $3$ matrix. We define the linear transformation $M:x \mapsto \bar{x}$ to be $\sigma^{\frac{9}{2p}}|E_{[0,1]^3}g(\bar{x})|=|E_{\tau}f(x)|$.
	Note that the image of $B_{N}$ under the transformation $M$ is a cylinder $C_{N}$ with dimensions $\sigma^{\frac{1}{2}} N \times \sigma^{\frac{1}{2}}N \times \sigma^{\frac{1}{2}} N \times \sigma N \times \sigma N \times \sigma N$. 
	Hence, by using a change of variables and changing back to the original variables, we have
	\begin{displaymath}
	\begin{split}
	\|E_{\tau}f\|_{L^p(B_{N})}^p & = \|E_{[0,1]^3}g\|_{L^p(C_{N})}^p
	\leq
	\sum_{B_{\sigma N} \cap C_{N} \neq \phi } \|E_{[0,1]^3}g\|_{L^p(B_{\sigma N})}^p
	\\ & \lesssim
	{D}_{S}({N\sigma},p)^p
	\sum\limits_{\substack{\theta \in \mathcal{P}_{N^{-1}} },\, \theta \subset \tau}\|E_{\theta}f\|_{L^p({w}_{B_{N}})}^p.
	\end{split}
	\end{displaymath}
	This completes the proof of Proposition 3.2.
\end{proof}

We note an easy lemma. This lemma follows by interpolating $L^2$ and $L^{\infty}$ estimates.

\begin{lem}[The trivial $l^p$ decoupling]
	Suppose that rectangles $\theta_{1},\ldots,\theta_{K}$ in $\mathbb{R}^{3}$ are mutually disjoint. Then for each $p \geq 2$, $g:[0,1]^3 \rightarrow \mathbb{C}$ and $K \geq 1$
	$$\|\sum_{i=1}^{K}E_{\theta_{i}}g\|_{p}^p \lesssim K^{p-2}\sum_{i=1}^{K}\|E_{\theta_{i}}g\|_{p}^p.$$
\end{lem}

Note that H\"{o}lder's inequality gives
\begin{displaymath}
\begin{split}
D_{S}(N,p) \lesssim N^{\frac{3}{2}(1-\frac{1}{p})}.
\end{split}
\end{displaymath}
Let $\gamma_{lin}(p) $ be the unique number such that
\begin{align*}
	\lim_{N \rightarrow \infty} D_{S}(N,p)N^{-\gamma_{lin}(p)-\epsilon}& =0,\; \mathrm{for}
	\; \mathrm{each} \; \epsilon>0,
	\\
	\limsup_{N \rightarrow \infty} D_{S}(N,p)N^{-\gamma_{lin}(p)+\epsilon}& =\infty,\; \mathrm{for}
	\; \mathrm{each} \; \epsilon>0.
\end{align*}
Similarly, let $\gamma_{bil}(p)$ be the unique number such that
\begin{align*}
	\lim_{N \rightarrow \infty} D_{bil}(N,p)N^{-\gamma_{bil}(p)-\epsilon}&=0,\; \mathrm{for}
	\; \mathrm{each} \; \epsilon>0,
	\\
	\limsup_{N \rightarrow \infty} D_{bil}(N,p)N^{-\gamma_{bil}(p)+\epsilon}&=\infty,\; \mathrm{for}
	\; \mathrm{each} \; \epsilon>0.
\end{align*}

\section{A decoupling clustered a variety}
\label{sec:4}
Let $S$ be a three-dimensional nondegenerate surface in $\R^6$. In this section, we always assume that there does not exist a hyperplane $L$ satisfying $\Phi(L)=0$. Let $Z$ be a 3-variety, i.e., $$Z=\{(x,y,z) \in \R^3:F(x,y,z)=0 \}$$ for some polynomial $F$ of degree three. For a set $V \subset \R^3$ and $N \geq 1$, we denote $C_{N}(V)$ a collection of all cubes in $\mathcal{P}_{N^{-2}}$ intersecting the set $V$.
The goal of this section is to prove the following proposition.
\begin{proposition}[A weak form of decoupling clustered a variety]
	For every $p\geq 2$ and $\epsilon>0$, there exist sufficiently large numbers $K$, $K_{1}$ and $K_{2}$ depending on $p$ and $\epsilon$ with $1 \ll K_{2} \ll K_{1} \ll K$ such that for every $f:[0,1]^3 \rightarrow \C$
	
	\begin{displaymath}
	\begin{split}
	\|\sum_{\alpha \in C_{K^{1/2}}(Z) }E_{\alpha}f \|_{L^p(w_{B_{K}})} \leq & \, C_{p,\epsilon} \biggl[
	K^{C\epsilon+\frac{3}{2}(\frac{1}{2}-\frac{1}{p})} (\sum_{\alpha \in \mathcal{P}_{K^{-1}}}\|E_{\alpha}f \|_{L^p(w_{B_{K}})}^p)^{\frac{1}{p}} 
	\\ &
	+K_{1}^{C\epsilon+\frac{3}{2}(\frac{1}{2}-\frac{1}{p})} (\sum_{\gamma \in \mathcal{P}_{K_{1}^{-1}} }\|\sum_{\alpha \in C_{K^{1/2}}(Z) } E_{\gamma \cap \alpha}f \|_{L^p(w_{B_{K}})}^p)^{\frac{1}{p}} 
	\\ &
	+K_{2}^{C\epsilon+\frac{3}{2}(\frac{1}{2}-\frac{1}{p})} (\sum_{\gamma \in \mathcal{P}_{K_{2}^{-1}}}\|\sum_{\alpha \in C_{K^{1/2}}(Z) } E_{ \gamma \cap \alpha}f \|_{L^p(w_{B_{K}})}^p)^{\frac{1}{p}} \biggr].
	\end{split}
	\end{displaymath}
	Moreover, the constant $C_{p,\epsilon}$ is stable under any slight translations of the variety $Z$.
\end{proposition}

The constants $K$, $K_{1}$ and $K_{2}$ will be explicitly determined in Section 5. Proposition 4.1 itself does not give a decoupling clustered the variety $Z$, but this inequality is sufficient for obtaining Theorem 1.1.

A key idea to prove Proposition 4.1 is to approximate the zero set Z by tangent planes. However, this set
does not have to be a manifold. Thus, we divide the set Z into two subsets: a manifold part and a singular
part. The $l^p$ decoupling associated with the manifold can be dealt with by the above approximation idea.
To handle the singular part, we make the singular part into a manifold by deleting a much singular subset.
In other words, we again divide the singular part into two sets: the manifold part and the much singular part. As
before, the $l^p$ decoupling associated with the manifold can be dealt with by the tangent plane approximation
argument. From the fact that Z is a zero set of a polynomial of degree three, the much singular part becomes a hyperplane. By using this observation, we can
handle the much singular part. This is an outline of the proof of Proposition 4.1.

We first need the uniform decoupling clustered arbitrary 1-variety. 

\begin{lemma}[The uniform decoupling clustered arbitrary 1-variety]
	Fix $p \geq 2$ and $\epsilon>0$. For every $f:[0,1]^3 \rightarrow \C$, hyperplane $L$ and $K\geq 1$, we have
	$$
	\|\sum_{\theta \in C_{K^{1/2}}(L) }E_{\theta}f \|_{L^p(w_{B_{K}})} \leq C_{p,\epsilon} K^{C\epsilon+\frac{3}{2}(\frac{1}{2}-\frac{1}{p})} (\sum_{\theta \in C_{K^{1/2}}(L) }\|E_{\theta}f \|_{L^p(w_{B_{K}})}^p)^{\frac{1}{p}}.
	$$
	Here, the constant $C_{p,\epsilon}$ is independent of a choice of $L$.
\end{lemma}

\begin{proof}
	We claim that there exists an $\epsilon_{1}>0$ such that
	\begin{equation}
	\inf_{M} \sup_{\xi \in M}| \Phi(\xi_{1},\xi_{2},\xi_{3})| > \epsilon_{1},
	\end{equation}
	where the infimum runs over all hyperplanes containing the origin. Suppose that such $\epsilon_{1}$ does not exist. Since the Grassmannian is sequentially compact, we can take a hyperplane $H$ such that $\Phi(H)=0$, which is a contradiction. Hence, $(4.1)$ holds true for some $\epsilon_{1}>0$.
	
	We can assume that $L$ intersects with $[0,1]^3$. By translation and rotation, we can further assume that $L$ is a $xy$-plane. By $(4.1)$ and a change of variables, we can write
	$$
	\Phi_{1}(\xi_{1},\xi_{2},0) = \xi_{1}^2 + O(\xi_{2}^2).
	$$
	By using the dimension reduction in \cite{BDG-Sharp-2017} and \cite{DGS-Sharp-2017}, and applying Lemma 2.4 in \cite{BD-Decouplings-2015}, we get the desired result.
\end{proof}
We need a rescaled version of Lemma 4.2. The proof of Lemma 4.3 is identical to that of Proposition 3.2.

\begin{lemma}
	Fix $p\geq2$ and $\epsilon>0$. For every $f:[0,1]^3 \rightarrow \C$, $K \geq 1$ and rectangular box $R$ with dimensions $\sim K^{-\frac{1}{2}} \times K^{-\frac{1}{2}} \times K^{-1}$, we have
	$$
	\|\sum_{\theta \in C_{K}(R) }E_{\theta}f \|_{L^p(w_{B_{K^2}})} \leq C_{p,\epsilon} K^{C\epsilon+\frac{3}{2}(\frac{1}{2}-\frac{1}{p})} (\sum_{\theta \in C_{K}(R) }\|E_{\theta}f \|_{L^p(w_{B_{K^2}})}^p)^{\frac{1}{p}}.
	$$
\end{lemma}
\begin{proof}
	Due to translation invariance, we can assume that $B_{K^2}=[0,K^2]^{6}$.
	By using a translation and a change of variables, we can assume that $R$ is contained in $[-2K^{-\frac{1}{2}},2K^{-\frac{1}{2}}]^3$. Define
	${h}(\xi)= 
	{K^{-\frac{3}{2}+\frac{9}{2p}}}{f}(\frac{\xi_{1}}{K^{1/2}},\frac{\xi_{2}}{K^{1/2}},\frac{\xi_{3}}{K^{1/2}})$.
	Then there exists a hyperplane $P$ such that the support of $h$ is contained in the $O(K^{-\frac{1}{2}})$-neighborhood of $P$. 
	Let $C_{N}=[0,K^{\frac{3}{2}}]^{3} \times [0,K]^3$. By Lemma 4.2,
	\begin{displaymath} 
	\begin{split}
	\|E_{[0,1]^3}f\|_{L^p({B_{K^2}})}^p& \
	=  \|E_{[0,1]^3}h\|_{L^p({C_{K}})}^p
	= 	\sum_{B_{K} \subset C_{K}} \|E_{[0,1]^3}h\|_{L^p(B_{K})}^p
	\\&
	\lesssim
	K^{C\epsilon+\frac{3}{2}(\frac{1}{2}-\frac{1}{p})p}
	\sum\limits_{\substack{\alpha \in \mathcal{P}_{K^{-1}}
	}}\|E_{\alpha}h\|_{L^p(	\sum_{B_{K} \subset C_{K}}{w}_{B_{K}})}^p 
	\\&
	\lesssim K^{C\epsilon+\frac{3}{2}(\frac{1}{2}-\frac{1}{p})p}\sum\limits_{\substack{\theta \in C_{K}(R)	}}\|E_{\theta}f\|_{L^p({w}_{B_{K^{2}}})}^p.
	\end{split}
	\end{displaymath}
	Now Lemma 3.1 gives the desired results.
\end{proof}

The next lemma is a decoupling clustered a manifold. The proof makes use of Lemma 4.3 and a tangent plane approximation argument.

\begin{lemma}[A decoupling clustered a manifold]
	Fix $p\geq2$ and $\epsilon>0$.
	Let $M$ be a graph of a manifold. For any $K \geq 1$ and $g:[0,1]^3 \rightarrow \C$, we have
	$$
	\| \sum_{\alpha \in C_{K^{1/2}}(M) } E_{\alpha}g \|_{L^p(w_{B_{K}})} \leq c(M)K^{C\epsilon+\frac{3}{2}(\frac{1}{2}-\frac{1}{p})}( \sum_{\alpha \in C_{K^{1/2}}(M)}  \|E_{\alpha}g \|_{L^p(w_{B_{K}})}^p)^{\frac{1}{p}}.
	$$
	Here, $c(M)$ is a constant depending on the principal curvatures of the manifold $M$. 
\end{lemma}

\begin{proof}
	We will use the following claim repeatedly.
	\\
	\textit{Claim.}
	Fix any number $K_{1}$ with $K_{1}^{-1} \geq K^{-\frac{1}{2}}$.  Let $\beta$ be a cube in $C_{K_{1}^{1/2}}(M)$. Then we have
	\begin{displaymath}
	\|\sum_{\alpha \in C_{K^{1/2}}(M) }E_{\alpha \cap \beta}g\|_{L^p({w}_{B_{K_{1}^2}})}
	\leq c(M) K_{1}^{C\epsilon+\frac{3}{2}(\frac{1}{2}-\frac{1}{p})}(\sum_{\gamma \in C_{K_{1}}(M) }\|
	\sum_{ \alpha \in C_{K^{1/2}}(M)}
	E_{\alpha \cap \gamma}g\|_{L^p({w}_{B_{K_{1}^2}})}^p)^{\frac{1}{p}}.
	\end{displaymath}
	We first prove the claim. Let $a$ be a point with $a \in M \cap \beta$. Take a tangent plane $T$ of $M$ at the point $a$. Then the intersection of $C_{K^{1/2}}(M)$ and $\beta$ is contained in a rectangular box with dimensions $C(M)(K_{1}^{-\frac{1}{2}} \times K_{1}^{-\frac{1}{2}} \times K_{1}^{-1})$. Applying Lemma 4.3 now completes the proof of Claim.
	
	We are now ready to prove Lemma 4.4. By H\"older's inequality
	\begin{displaymath}
	\begin{split}
	\| \sum_{\alpha \in C_{K^{1/2}}(M) } E_{\alpha}g \|_{L^p({B_{K}})}^p \leq  CK^{\epsilon}\sum_{B_{K^{4\epsilon}} \subset B_{K} }\sum_{\beta \in C_{K^{\epsilon}}(M)} \| \sum_{\alpha \in C_{K^{1/2}}(M) } E_{\alpha \cap \beta}g \|_{L^p(B_{K^{4\epsilon}})}^p.
	\end{split}
	\end{displaymath}
	By applying Claim with $K_{1}=K^{2\epsilon}$ and summing over cubes $B_{K^{4\epsilon}} \subset B_{K^{8\epsilon}}$, the above term is bounded by
	$$
	\leq C(M) K^{\epsilon}K^{C\epsilon^2p +\frac{3}{2}(\frac{1}{2}-\frac{1}{p})2\epsilon p}\sum_{B_{K^{8\epsilon}} \subset B_{K} }\sum_{\beta \in C_{K^{2\epsilon}}(M)} \| \sum_{\alpha \in C_{K^{1/2}}(M) } E_{\alpha \cap \beta}g \|_{L^p(w_{B_{K^{8\epsilon}}})}^p.
	$$
	We again apply Claim with $K_{1}=K^{4\epsilon}$ and sum over cubes $B_{K^{8\epsilon}} \subset B_{K^{16\epsilon}}$. Then
	$$
	\leq C(M)^2 K^{\epsilon}K^{C\epsilon^2p +\frac{3}{2}(\frac{1}{2}-\frac{1}{p})p\epsilon (2+2^2)}\sum_{B_{K^{16\epsilon}} \subset B_{K} }\sum_{\beta \in C_{K^{4\epsilon}}(M)} \| \sum_{\alpha \in C_{K^{1/2}}(M) } E_{\alpha \cap \beta}g \|_{L^p(w_{B_{K^{16\epsilon}}})}^p.
	$$
	We repeat this process until the side length of cubes becomes $K$. Then the above term is bounded by
	$$
	\leq C(M)^{10\log \epsilon^{-1}}K^{C\epsilon+\frac{3}{2}(\frac{1}{2}-\frac{1}{p})p\epsilon(2+2^2+2^3+\cdots+\frac{\epsilon^{-1}}{2})} \sum_{\alpha \in C_{K^{1/2}}(M)}\|E_{\alpha}g \|_{L^p(w_{B_{K}})}^p.
	$$
	We take $c(M)=C(M)^{10\log \epsilon^{-1}}$. The exponent of $K$ in the above term is less than $C\epsilon+\frac{3}{2}(\frac{1}{2}-\frac{1}{p})p$. It suffices now to apply Lemma 3.1.
\end{proof}

\subsubsection*{Proof of Proposition 4.1.}

Let $Z+b_{0}$ be the translation of $Z$ for some $b_{0} \in \R^3$.
Fix sufficiently large $K_{2}$. Constants $K$ and $K_{1}$ will be determined later.
Let $Z$ be the zero set of a polynomial $F$ of degree three. Since $F$ is a nonzero function, one of the three functions $\frac{\partial F}{\partial \xi_{1}},\frac{\partial F}{\partial \xi_{2}},\frac{\partial F}{\partial \xi_{3}}$ is not identically zero. We call this function $F^{(1)}$. Note that the function $F^{(1)}$ is a polynomial of degree two. Define
$$U^{(1)}=
\bigcup_{{Q \in \mathcal{P}_{K_{1}^{-1}} \,:\, \exists \xi  \in 5Q-b_{0} \; \mathrm{s.t} \; F^{(1)}(\xi)=0 }}Q.$$
Since $F^{(1)}$ is a nonzero function, one of the three functions $\frac{\partial F^{(1)}}{\partial \xi_{1}},\frac{\partial F^{(1)}}{\partial \xi_{2}},\frac{\partial F^{(1)}}{\partial \xi_{3}}$ is not identically zero. We call this function $F^{(2)}$, and define a set $U^{(2)}$ by
$$U^{(2)}=
\bigcup_{{Q \in \mathcal{P}_{K_{2}^{-1}} \,:\, \exists \xi \in 5Q-b_{0} \; \mathrm{s.t} \;  F^{(2)}(\xi)=0 }}Q.$$
Observe that $U^{(2)}$ is contained in the $30 K_{2}^{-\frac{1}{2}}$-neighborhood of some hyperplane in $\mathbb{R}^3$ because the function $F^{(2)}$ is a nonzero polynomial of degree one. We will deal with the decoupling associated with the set $U^{(2)}$ by using Lemma 4.2. 


Now we start the proof of Proposition 4.1.
By the triangle inequality,
\begin{displaymath}
\begin{split}
\|\sum_{\substack{\alpha \in C_{K^{1/2}}(Z)  }}	E_{\alpha}f\|_{L^p({B_{K}})}^p & \lesssim \|\sum\limits_{\substack{\alpha \in C_{K^{1/2}}(Z), \\ \alpha \subset U^{(2)}}}E_{\alpha}f\|_{L^p({B_{K}})}^p
+
\|\sum\limits_{\substack{\alpha \in C_{K^{1/2}}(Z), \\ \alpha \cap U^{(2)}= \phi }}E_{\alpha}f\|_{L^p({B_{K}})}^p
\\ &
\lesssim 
\sum_{{B_{K_{2}} \subset B_{K }}} \|\sum\limits_{\substack{\alpha \in C_{K^{1/2}}(Z), \\ \alpha \subset U^{(2)}}}E_{\alpha}f\|_{L^p({B_{K_{2}}})}^p
+
\|\sum\limits_{\substack{\alpha \in C_{K^{1/2}}(Z), \\ \alpha \cap U^{(2)}= \phi }}E_{\alpha}f\|_{L^p({B_{K}})}^p \nonumber
.
\end{split}
\end{displaymath}
By applying Lemma 4.2 to the first term on the right hand side and using the triangle inequality, the bound terms are bounded by
\begin{gather*}
	\lesssim
	K_{2}^{C\epsilon+\frac{3}{2}(\frac{1}{2}-\frac{1}{p})p}
	\sum_{B_{K_{2}} \subset B_{K}}
	\sum_{\beta \in \mathcal{P}_{K_{2}^{-1}}}
	\|\sum\limits_{\substack{\alpha \in C_{K^{1/2}}(Z), \\ \alpha \subset U^{(2)}}}E_{\alpha \cap \beta }f\|_{L^p({w_{B_{K_{2}}}})}^p
	+
	\|\sum\limits_{\substack{\alpha \in C_{K^{1/2}}(Z), \\ \alpha \cap U^{(2)}= \phi }}E_{\alpha}f\|_{L^p({B_{K}})}^p 
	\\
	\lesssim
	K_{2}^{C\epsilon+\frac{3}{2}(\frac{1}{2}-\frac{1}{p})p}
	\sum_{\beta \in \mathcal{P}_{K_{2}^{-1}}}
	\|\sum\limits_{\substack{\alpha \in C_{K^{1/2}}(Z), \\ \alpha \subset U^{(2)}}}E_{\alpha \cap \beta }f\|_{L^p({w_{B_{K}}})}^p
	+
	\|\sum\limits_{\substack{\alpha \in C_{K^{1/2}}(Z), \\ \alpha \cap U^{(2)}= \phi }}E_{\alpha}f\|_{L^p({B_{K}})}^p 
	\\
	\lesssim
	K_{2}^{C\epsilon+\frac{3}{2}(\frac{1}{2}-\frac{1}{p})p}
	\sum_{\beta \in \mathcal{P}_{K_{2}^{-1}}}
	\|E_{\beta }f\|_{L^p({w_{B_{K}}})}^p 
	+
	K_{2}^{C\epsilon+\frac{3}{2}(\frac{1}{2}-\frac{1}{p})p}
	\sum_{\beta \in \mathcal{P}_{K_{2}^{-1}}}
	\|\sum\limits_{\substack{\alpha \in \mathcal{P}_{K^{-1}}, \\ \alpha \cap U^{(2)}= \phi }}E_{\alpha \cap \beta }f\|_{L^p({w_{B_{K}}})}^p
	\\
	+
	\|\sum\limits_{\substack{\alpha \in C_{K^{1/2}}(Z), \\ \alpha \cap U^{(2)}= \phi }}E_{\alpha}f\|_{L^p({B_{K}})}^p.\nonumber
\end{gather*}
By the triangle inequality, the above terms are bounded by

\begin{gather*}
\lesssim K_{2}^{C\epsilon+\frac{3}{2}(\frac{1}{2}-\frac{1}{p})p}
\sum_{\beta \in \mathcal{P}_{K_{2}^{-1}}}
\|E_{\beta }f\|_{L^p({w_{B_{K}}})}^p 
\\
+
K_{2}^{C\epsilon+\frac{3}{2}(\frac{1}{2}-\frac{1}{p})p}
\sum_{\beta \in \mathcal{P}_{K_{2}^{-1}}}
\|\sum\limits_{\substack{\alpha \in C_{K^{1/2}}(Z), \\ \alpha \cap (U^{(1)} \cup U^{(2)})= \phi }}E_{\alpha \cap \beta }f\|_{L^p({w_{B_{K}}})}^p
\\
+
K_{2}^{C\epsilon+\frac{3}{2}(\frac{1}{2}-\frac{1}{p})p}
\sum_{\beta \in \mathcal{P}_{K_{2}^{-1}}}
\|\sum\limits_{\substack{\alpha \in C_{K^{1/2}}(Z), \\ \alpha \subset (U^{(1)}\setminus U^{(2)})= \phi }}E_{\alpha \cap \beta }f\|_{L^p({w_{B_{K}}})}^p
\\
+\|\sum\limits_{\substack{\alpha \in C_{K^{1/2}}(Z), \\ \alpha \cap (U^{(1)} \cup U^{(2)})= \phi }}E_{\alpha}f\|_{L^p({B_{K}})}^p 
\\
+ 
\|\sum\limits_{\substack{\alpha \in C_{K^{1/2}}(Z),\\  \alpha \subset (U^{(1)} \setminus U^{(2)}) }}E_{\alpha}f\|_{L^p({B_{K}})}^p.
\end{gather*}
We take $K$ and $K_{1}$ large enough so that $K^{\epsilon},K_{1}^{\epsilon} \geq K_{2}^{\frac{3}{2}+\frac{3}{2}(\frac{1}{2}-\frac{1}{p})p}$.
It suffices now to show the following two inequalities.

\subsubsection*{Claim} If $K$ and $K_{1}$ are sufficiently large, then
		\begin{gather*}
		\|\sum\limits_{\substack{\alpha \in C_{K^{1/2}}(Z), \\ \alpha \subset (U^{(1)} \setminus U^{(2)}) }}E_{\alpha }f\|_{L^p({B_{K}})}^p 
		\leq  K_{1}^{C\epsilon+\frac{3}{2}(\frac{1}{2}-\frac{1}{p})p}\sum\limits_{\gamma \in \mathcal{P}_{K_{1}^{-1}}}\|\sum\limits_{\substack{\alpha \in C_{K^{1/2}}(Z), \\ \alpha \subset (U^{(1)} \setminus U^{(2)})}}E_{\alpha \cap \gamma}f\|_{L^p({w}_{B_{K}})}^p,
		\\ 
		\|\sum\limits_{\substack{\alpha \in C_{K^{1/2}}(Z), \\ \alpha \cap (U^{(1)} \cup U^{(2)})= \phi }}E_{\alpha}f\|_{L^p({B_{K}})}^p \leq K^{C\epsilon+\frac{3}{2}(\frac{1}{2}-\frac{1}{p})p} \sum_{\theta \in C_{K^{1/2}}(Z)}\|E_{\theta}f \|_{L^p(w_{B_{K}})}^p.
		\end{gather*}

\subsubsection*{Proof of Claim}
	First, we prove the first inequality. By the construction of the sets $U^{(i)}$ and the implicit function theorem, the set $U^{(1)} \setminus U^{(2)}$ is contained in a finite union of the $O(K_{1}^{-\frac{1}{2}})$-neighborhood of a manifold whose principal curvatures depend on the constant $K_{2}$. By H\"{o}lder's inequality, we can pretend that the set $U^{(1)} \setminus U^{(2)}$ is contained in the $O(K_{1}^{-\frac{1}{2}})$-neighborhood of a manifold $M$. We apply Lemma 4.4. Then
	$$
	\|\sum\limits_{\substack{\alpha \in C_{K^{1/2}}(Z), \\ \alpha \subset (U^{(1)} \setminus U^{(2)}) }}E_{\alpha }f\|_{L^p({B_{K}})}^p 
	\leq c(K_{2}) K_{1}^{C\epsilon+\frac{3}{2}(\frac{1}{2}-\frac{1}{p})p}\sum\limits_{\gamma \in \mathcal{P}_{K_{1}^{-1}}}\|\sum\limits_{\substack{\alpha \in C_{K^{1/2}}(Z), \\ \alpha \subset (U^{(1)} \setminus U^{(2)})}}E_{\alpha \cap \gamma}f\|_{L^p({w}_{B_{K}})}^p.
	$$
	We fix $K_{1}$ satisfying $c(K_{2}) \leq K_{1}^{\epsilon}$. Lemma 3.1 now completes the proof.
	
	We now prove the second inequality. Because $C_{K^{1/2}}(Z) \setminus (U^{(1)} \cup U^{(2)})$ is a manifold, we can follow the same argument. By Lemma 4.4, we have
	$$
	\|\sum\limits_{\substack{\alpha \in C_{K^{1/2}}(Z), \\ \alpha \cap (U^{(1)} \cup U^{(2)})=\phi }}E_{\alpha }f\|_{L^p({B_{K}})}^p 
	\leq c(K_{1},K_{2}) K^{C\epsilon+\frac{3}{2}(\frac{1}{2}-\frac{1}{p})p}\sum\limits_{\gamma \in \mathcal{P}_{K^{-1}}}\|\sum\limits_{\substack{\alpha \in C_{K^{1/2}}(Z), \\ \alpha \subset (U^{(1)} \setminus U^{(2)})}}E_{\alpha \cap \gamma}f\|_{L^p({w}_{B_{K}})}^p.
	$$
	We fix $K$ satisfying $c(K_{1},K_{2}) \leq K^{\epsilon}$. Lemma 3.1 now completes the proof.
	\qed


\section{Linear versus bilinear decoupling for surfaces}
\label{sec:5}

The goal of this section is to prove Theorem 5.1, which means that the bilinear $l^p$ decoupling implies the linear $l^p$ decoupling. Since the bilinear $l^p$ decoupling is much easier than the linear $l^p$ decoupling, the theorem is useful. Since our surface is a submanifold of half the ambient dimension, the multilinear $l^p$ decoupling is the same as the bilinear decoupling.

Let $S$ be a three-dimensional nondegenerate surface in $\mathbb{R}^{6}$. Let $Z=\{ \xi \in \R^3:J(\xi,0)=0	\}$.


\begin{thm}
	\leavevmode
	Let $p \geq 2$ and $\epsilon>0$.
	\begin{enumerate}
		\item If there is some two-dimensional plane $L$ in $\R^3$ satisfying $\Phi(L) =0$, then there exists $\nu$ such that for every $N \geq 1$
		$$
		D_{S}(N,p) \leq C_{p,\epsilon}N^{\epsilon}\sup_{1<M<N} \bigg[(\frac{N}{M})^{1-\frac{2}{p}}D_{bil}(M,p,\nu) \biggr].
		$$
		\item If such $L$ does not exist, then there exists $\nu$ such that for every $N \geq 1$ 
			$$D_{S}(N,p) \leq C_{p,\epsilon}N^{\epsilon}\sup_{1<M<N}
			\biggl[
			(\frac{N}{M})^{\frac{3}{2}(\frac{1}{2}-\frac{1}{p})}D_{bil}(M,p,\nu) \biggr].$$
	\end{enumerate}	
\end{thm}


The proof of (1) in Theorem 5.1 does not require any novelty. We first note that if there exists a hyperplane $L$ in $\mathbb{R}^{3}$ satisfying $\Phi(L)=0$, then we can assume that $$S=\{(\xi_{1},\xi_{2},\xi_{3},\xi_{1}^2,2\xi_{1}\xi_{2},2\xi_{1}\xi_{3}): (\xi_{1},\xi_{2},\xi_{3}) \in [0,1]^3 \}$$  by using a change of variables. Thus, we have 
$ J(p_{1},p_{2})=8|\xi_{1}-\eta_{1}|^3
$
for any $p_{1}=(\xi_{1},\xi_{2},\xi_{3})$ and $p_{2}=(\eta_{1},\eta_{2},\eta_{3})$. Since the set $Z$ is a hyperplane, we can easily prove (1) of Theorem 5.1 just by following Bourgain and Demeter's argument in \cite{BD-The-2015}.

Note that the exponent of $(\frac{N}{M})$ in (2) of Theorem 5.1 is identical to the exponents of $K,K_{1},K_{2}$ in Proposition 4.1. This is because this term comes from a decoupling clustered near a variety.

To prove Theorem 5.1, we will first prove Proposition 5.2.


\begin{propo}
	\leavevmode
	Let $p\geq2$ and $\epsilon>0$.
	\begin{enumerate}
		\item If there is some two-dimensional plane  $L$ in $\R^3$ satisfying $\Phi(L)=0$, then there exist sufficiently large number $K$ and some number $C_{K}$ such that for any $f:[0,1]^3 \rightarrow \C$ and $N \geq K^2$
		\begin{displaymath}
		\begin{split}
		\|E_{[0,1]^3}f\|_{L^p({B_{N}})}^p & \leq C_{p}K^{(1-\frac{2}{p})p} \sum_{\alpha \in \mathcal{P}_{K^{-1}} }\|E_{\theta}f\|_{L^p({w}_{B_{N}})}^p
		\\ &
		+C_{p}K^{100p}D_{bil}(N,p,C_{K})^p
		\sum_{\theta \in \mathcal{P}_{N^{-1}}}\|E_{\theta}f\|_{L^p({w}_{B_{N}})}^p.
		\end{split}
		\end{displaymath}
		\item If such $L$ does not exist, then
		there exist sufficiently large numbers $K,K_{1}$ and $K_{2}$ with $1 \ll  K_{2} \ll K_{1} \ll K$ and some number $C_{K}$ such that for any $f: [0,1]^3 \rightarrow \mathbb{C}$ and $N \geq K^2$
		\begin{displaymath}
		\begin{split}
		\|E_{[0,1]^3}f\|_{L^p({B_{N}})}^p & \leq 
		C_{p,\epsilon}
		K_{2}^{C\epsilon+ \frac{3}{2}(\frac{1}{2}-\frac{1}{p})p}\sum_{\beta \in \mathcal{P}_{K_{2}^{-1}} }\|
		E_{\beta}f\|_{L^p({w}_{B_{N}})}^p
		\\&
		+C_{p,\epsilon} K_{1}^{C\epsilon+\frac{3}{2}(\frac{1}{2}-\frac{1}{p})p} \sum_{\beta \in \mathcal{P}_{K_{1}^{-1}}} \|E_{\beta}f\|_{L^p(w_{B_{N}})}^p
		\\&
		+C_{p,\epsilon}K^{C\epsilon+\frac{3}{2}(\frac{1}{2}-\frac{1}{p})p}
		\sum_{\beta \in \mathcal{P}_{K^{-1}}}\|E_{ \beta}f\|_{L^p({w}_{B_{N}})}^p
		\\&
		+C_{p,\epsilon}K^{C\epsilon+{3}(\frac{1}{2}-\frac{1}{p})p}
		\sum_{\alpha \in \mathcal{P}_{K^{-2}}}\|E_{ \alpha}f\|_{L^p({w}_{B_{N}})}^p
		\\ &
		+C_{p,\epsilon}K^{100p}D_{bil}(N,p,C_{K})^p
		\sum_{\theta \in \mathcal{P}_{N^{-1}}}\|E_{\theta}f\|_{L^p({w}_{B_{N}})}^p.
		\end{split}
		\end{displaymath}
	\end{enumerate}
\end{propo}

The constants $K,K_{1}$ and $K_{2}$ will be determined in the proof of Theorem 5.1. The proof of Proposition 5.2 is very similar to that of Proposition 5.2 in \cite{BD-The-2015}.

\subsubsection*{Proof of Proposition 5.2.}
We first prove $(2)$ of Proposition 5.2.
Due to translation invariance, we can assume that $B_{N}=[0,N]^{6}$. We will follow the standard formalism in \cite{BG-Bounds-2011}. Fix a cube $B_{K}(a)$ in $B_{N}$. We take a Schwartz function $\eta$ on $\mathbb{R}^{6}$, with $\hat{\eta}(x)=1$ on $[-2,2]^{6}$ and $\hat{\eta}(x)=0$ outside $[-4,4]^{6}$. We also take the function $\zeta_{B_{K}(a)}(x)=K^{-6}w_{B_{K}(a)}^{100}(x)$ so that $\|\zeta_{B_{K}(a)}	\|_{L^1(\mathbb{R}^{6})} \sim	1$. If $a=0$, we sometimes write $\zeta_{{K}}$ instead of $\zeta_{B_{K}(0)}$. For each cube $\alpha=b_{\alpha}+[0,K^{-1}]^{3} \in \mathcal{P}_{K^{-2}}$ with some $b_{\alpha}\in [0,1]^3$, we define $$\eta_{K,\alpha}(x)=K^{-6}e(x \cdot (b_{\alpha},\Phi(b_{\alpha})))\eta(\frac{x}{K}).$$ 
By an application of Young's inequality
\begin{displaymath}
\begin{split}
\|E_{\alpha}f(x-\cdot)\eta_{K,\alpha}(\cdot)\|_{L^1} &\leq \|E_{\alpha}f(x-\cdot)\eta_{K,\alpha}(\cdot)\|_{L^{\infty}}^{\frac{1}{2}}
\|E_{\alpha}f(x-\cdot)\eta_{K,\alpha}(\cdot)\|_{L^{\frac{1}{2}}}^{\frac{1}{2}}
\\& \lesssim {K^{-{3}}}
\|E_{\alpha}f(x-\cdot)\eta_{K,\alpha}(\cdot)\|_{L^{1}}^{\frac{1}{2}}
\|E_{\alpha}f(x-\cdot)\eta_{K,\alpha}(\cdot)\|_{L^{\frac{1}{2}}}^{\frac{1}{2}}.
\end{split}
\end{displaymath}
Hence, we have
\begin{displaymath}
|E_{\alpha}f(x)|=|E_{\alpha}f*\eta_{K,\alpha}(x)| \lesssim \bigl(\int_{\mathbb{R}^{6}}|E_{\alpha}f(x-y)|^{\frac{1}{2}}\frac{1}{K^{6}}|\eta(\frac{y}{K})|^{\frac{1}{2}}\,dy\bigr)^2.
\end{displaymath}
Define 
$$c_{\alpha}(B_{K}(a))=(\int_{\mathbb{R}^{6}}|E_{\alpha}f(y)|^{\frac{1}{2}}\zeta_{B_{K}(a)}(y)\,dy\bigr)^2.$$
Note that for any $x\in B_{K}(a)$
\begin{displaymath}
\begin{split}
|E_{\alpha}f(x)| & \lesssim 
\bigl(\int_{\mathbb{R}^{6}}|E_{\alpha}f(y)|^{\frac{1}{2}}\zeta_{B_{K}(0)}(x-y)\,dy\bigr)^2  
\lesssim
c_{\alpha}(B_{K}(a)) 
\\&
\lesssim \int |E_{\alpha}f(y)|\zeta_{B_{K}(a)}(y)\,dy
\lesssim
\int_{\mathbb{R}^{6}}|E_{\alpha}f(x-y)|\zeta_{B_{K}(0)}(y)\,dy.
\end{split}
\end{displaymath}

Let $\alpha^{(1)} \in \mathcal{P}_{K^{-2}}$ be a cube maximizing the value $c_{\alpha}(B_{K}(a))$. 
There are two possibilities. 

(Case 1: a transverse case) consider the case that  there is some cube $\alpha^{(2)} \in \mathcal{P}_{K^{-2}} $ such that $\alpha^{(2)} \cap (C_{K^{1/2}}(Z)+b_{\alpha^{(1)}}) = \phi$  and $c_{\alpha^{(2)}}(B_{K}(a)) \geq K^{-3}c_{\alpha^{(1)}}(B_{K}(a))$. Note that $\alpha^{(1)}$ and $\alpha^{(2)}$ are $C_{K}$-transverse. For any $x \in B_{K}(a)$ we have
$$|E_{[0,1]^3}f(x)| = |\sum_{\alpha} E_{\alpha}f(x)| \leq  K^{\frac{9}{2}}(c_{\alpha^{(1)}}(B_{K}(a))c_{\alpha^{(2)}}(B_{K}(a)))^{\frac{1}{2}},$$ 
and we also have
\begin{gather*}  
	|c_{\alpha^{(1)}}(B_{K}(a))c_{\alpha^{(2)}}(B_{K}(a))|^{\frac{1}{2}}
	\lesssim \iint_{\mathbb{R}^{6} \times \mathbb{R}^{6} } \prod_{i=1}^{2} |E_{\alpha^{(i)}}f(x-y_{i})|^{\frac{1}{2}}\zeta_{K}(y_{i})\,dy_{1}dy_{2}.
\end{gather*}
Raising to the $p$ power, integrating on the cube $B_{K}(a)$ and H\"{o}lder's inequality give
\begin{displaymath}
\begin{split}
	\|E_{[0,1]^3}f\|_{L^p(B_{K}(a))}^p & 
	\lesssim 
	K^{6p}\iint_{\mathbb{R}^{6} \times \mathbb{R}^{6}} \|\prod_{i=1}^{2}| E_{\alpha^{(i)}}f(\cdot-y_{i})|^{\frac{1}{2}}\zeta_{K}(y_{i})\|_{L^p(B_{K}(a))}^{p}
	\,dy_{1}dy_{2}
	\\ &
	\lesssim 
	K^{6p}\iint
	\sum_{
		\substack{
			\alpha_{1}, \alpha_{2} \in \mathcal{P}_{K^{-2}} \\(\alpha_{1},\alpha_{2}):C_{K}-\mathrm{trans}
	}}
	\|\prod_{i=1}^{2}|E_{\alpha_{i}}f(\cdot-y_{i})|^{\frac{1}{2}}\zeta_{K}(y_{i})\|_{L^p(B_{K}(a))}^{p}
	\,dy_{1}dy_{2}.
\end{split}
\end{displaymath}

(Case 2: a non-transverse case) Suppose that Case 1 does not occur. If a cube $\alpha \in \mathcal{P}_{K^{-2}}$ satisfies $\alpha \cap (C_{K^{1/2}}(Z)+b_{\alpha^{(1)}}) = \phi$, then $ c_{\alpha}(B_{K}(a)) \leq K^{-3}c_{\alpha^{(1)}}(B_{K}(a)).$
Thus, for any $x\in B_{K}(a)$ we have
$$|E_{[0,1]^3}f(x)| \lesssim |\sum_{\beta \in \mathcal{P}_{K^{-1}}:\beta \cap (C_{K^{1/2}}(Z)+b_{\alpha^{(1)}}) \neq \phi}	E_{\beta}f(x)| +c_{\alpha^{(1)}}(B_{K}(a)).$$
By raising to the power $p$ and integrating on the cube $B_{K}(a)$, we have
\begin{displaymath}
\begin{split}
\|E_{[0,1]^3}f\|_{L^p(B_{K}(a))}^p &\lesssim \|\sum_{\beta \in \mathcal{P}_{K^{-1}}:\beta \cap (C_{K^{1/2}}(Z)+b_{\alpha^{(1)}}) \neq \phi}	E_{\alpha}f(x)	\|_{L^p(B_{K}(a))}^p \\&+ |c_{\alpha^{(1)}}(B_{K}(a))|^{p}|B_{K}(a)|.
\end{split}
\end{displaymath}
The second term can be easily handled; by H\"{o}lder's inequality
$$|c_{\alpha^{(1)}}(B_{K}(a))|^{p}|B_{K}| \lesssim 
\int |E_{\alpha^{(1)}}f(y)|^p
|B_{K}| \zeta_{B_{K}(a)}(y) \,dy
\lesssim \|E_{\alpha^{(1)}}f\|_{L^p({w}_{B_{K}(a)})}^p.$$
To handle the first term, we need the following inequality.
\begin{equation}
\begin{split}
\|E_{[0,1]^3}g	\|_{L^p(B_{K}(a))}^p &  
\leq C_{p,\epsilon}
K_{2}^{C\epsilon+ \frac{3}{2}(\frac{1}{2}-\frac{1}{p})p}\sum_{\gamma \in \mathcal{P}_{K_{2}^{-1}} }\|
E_{\gamma}g\|_{L^p({w}_{B_{K}}(a))}^p
\\&
+C_{p,\epsilon} K_{1}^{C\epsilon+\frac{3}{2}(\frac{1}{2}-\frac{1}{p})p} \sum_{\gamma \in \mathcal{P}_{K_{1}^{-1}} } \|E_{\gamma}g\|_{L^p(w_{B_{K}}(a))}^p
\\&
+C_{p,\epsilon}K^{C\epsilon+\frac{3}{2}(\frac{1}{2}-\frac{1}{p})p}
\sum_{\gamma \in \mathcal{P}_{K^{-1}}}\|E_{ \gamma}g\|_{L^p({w}_{B_{K}}(a))}^p
,
\end{split}
\end{equation}
where $g(\xi)=\sum\limits_{\substack{\beta \in \mathcal{P}_{K^{-1}} : \alpha \cap (C_{K^{1/2}}(Z)+b_{\alpha^{(1)}})) \neq \phi}} 1_{\alpha}(\xi)f(\xi)$. This inequality immediately follows from Proposition 4.1. Since we are dealing with the second scenario,
\begin{displaymath}
\begin{split}
\|E_{[0,1]^3}g	\|_{L^p(B_{K}(a))}^p &  \leq C_{p,\epsilon}
K_{2}^{C\epsilon+ \frac{3}{2}(\frac{1}{2}-\frac{1}{p})p}\sum_{\gamma \in \mathcal{P}_{K_{2}^{-1}} }\|
E_{\gamma}f\|_{L^p({w}_{B_{K}}(a))}^p
\\&
+C_{p,\epsilon} K_{1}^{C\epsilon+\frac{3}{2}(\frac{1}{2}-\frac{1}{p})p} \sum_{\gamma \in \mathcal{P}_{K_{1}^{-1}} } \|E_{\gamma}f\|_{L^p(w_{B_{K}}(a))}^p
\\&
+C_{p,\epsilon}K^{C\epsilon+\frac{3}{2}(\frac{1}{2}-\frac{1}{p})p}
\sum_{\gamma \in \mathcal{P}_{K^{-1}}}\|E_{ \gamma}f\|_{L^p({w}_{B_{K}}(a))}^p
\\&
+C_{p,\epsilon}K^{C\epsilon+{3}(\frac{1}{2}-\frac{1}{p})p} \|E_{\alpha^{(1)}}f\|_{L^p(w_{B_{K}}(a))}^p.
\end{split}
\end{displaymath}

To summarize, in either case, we have

\begin{displaymath}
\begin{split}
	\|E_{[0,1]^3}f\|_{L^p(B_{K}(a))}^{p} &
\leq C_{p,\epsilon}
K_{2}^{C\epsilon+ \frac{3}{2}(\frac{1}{2}-\frac{1}{p})p}\sum_{\gamma \in \mathcal{P}_{K_{2}^{-1}} }\|
E_{\gamma}f\|_{L^p({w}_{B_{K}}(a))}^p
\\&
+C_{p,\epsilon} K_{1}^{C\epsilon+\frac{3}{2}(\frac{1}{2}-\frac{1}{p})p} \sum_{\gamma \in \mathcal{P}_{K_{1}^{-1}} } \|E_{\beta}f\|_{L^p(w_{B_{K}}(a))}^p
\\&
+C_{p,\epsilon}K^{C\epsilon+\frac{3}{2}(\frac{1}{2}-\frac{1}{p})p}
\sum_{\gamma \in \mathcal{P}_{K^{-1}}}\|E_{ \gamma}f\|_{L^p({w}_{B_{K}}(a))}^p
\\&
+C_{p,\epsilon}K^{C\epsilon+{3}(\frac{1}{2}-\frac{1}{p})p}
\sum_{\alpha \in \mathcal{P}_{K^{-2}}}\|E_{ \alpha}f\|_{L^p({w}_{B_{K}}(a))}^p
\\&
+C_{p}K^{6p}\iint  \sum_{
	\substack{
		\alpha_{1}, \alpha_{2} \in \mathcal{P}_{K^{-2}} \\(\alpha_{1},\alpha_{2}):C_{K}-\mathrm{trans}
}} \|\prod_{i=1}^{2}|E_{\alpha_{i}}f(\cdot-y_{i})|^{\frac{1}{2}}\zeta_{K}(y_{i})\|_{L^p(B_{K}(a))}^p \,dy_{1}dy_{2}.
\end{split}
\end{displaymath}
It suffices now to sum over $B_{K} \subset {B_{N}}$ and use the definition of $D_{bil}(N,p,\nu)$ and Fubini's theorem.

The proof of (1) of Proposition 5.2 is identical to that of (2) of Proposition 5.2 except that instead of $(5.1)$ we use the following inequality.
\begin{displaymath}
\begin{split}
\|\sum\limits_{\substack{\beta \in \mathcal{P}_{K^{-1}} \\ \beta \cap (C_{K^{1/2}}(Z)+b_{\alpha^{(1)}}) \neq \phi}}	E_{\beta}f(x)	\|_{L^p(B_{K}(a))}^p \leq C_{p} K^{p-2}
\sum_{\beta \in \mathcal{P}_{K^{-1}}}
\|E_{\beta}f\|_{L^p({w}_{B_{K}(a)})}^p.
\end{split}
\end{displaymath}
This inequality follows from Lemma 3.1 and the fact that the number of the cubes $\beta$ with side length $K^{-\frac{1}{2}}$ intersecting the hyperplane $C_{K^{1/2}}(Z)$ is $O(K)$.
\qed

To iterate Proposition 5.2, we need a rescaled version of it.
The proof of Proposition 5.3 is similar to that of Proposition 3.2.
\begin{propo}
	\leavevmode
	Let $p \geq 2$ and $\epsilon>0$.
	\begin{enumerate}
		\item If there is some two-dimensional plane $L$ in $\R^3$ satisfying $\Phi(L)=0$, then there exist sufficiently large number $K$ and some number $C_{K}$ such that for any $f:[0,1]^3 \rightarrow \C$, any numbers $N,t$ with $N \geq K^2$ and $\frac{K^2}{N} \leq t \leq 1$, and any $\alpha \in \mathcal{P}_{t}$
		\begin{displaymath}
		\begin{split}
		\|E_{\alpha}f\|_{L^p(w_{B_{N}})}^p & \leq C_{p}K^{(1-\frac{2}{p})p} \sum_{\theta \in \mathcal{P}_{tK^{-1}} }\|E_{\theta}f\|_{L^p({w}_{B_{N}})}^p
		\\ &
		+C_{p}K^{100p}D_{bil}(Nt,p,C_{K})^p
		\sum_{\theta \in \mathcal{P}_{N^{-1}}}\|E_{\theta}f\|_{L^p({w}_{B_{N}})}^p.
		\end{split}
		\end{displaymath}
		
		\item
		If such $L$ does not exist, then
		there exist sufficiently large numbers $K,K_{1}$ and $K_{2}$ with $1 \ll  K_{2} \ll K_{1} \ll K$ and some number $C_{K}$ such that for any $f: [0,1]^3 \rightarrow \mathbb{C}$, any numbers $N$ and $t$ with $N \geq K^2$, $\frac{K^2}{N} \leq t \leq 1$  and any $\alpha \in \mathcal{P}_{t}$
		\begin{displaymath}
		\begin{split}
		\|E_{\alpha}f\|_{L^p({B_{N}})}^p & \leq 
		C_{p,\epsilon}
		K_{2}^{C\epsilon+ \frac{3}{2}(\frac{1}{2}-\frac{1}{p})p}\sum_{\theta \in \mathcal{P}_{tK_{2}^{-1}} }\|
		E_{\theta}f\|_{L^p({w}_{B_{N}})}^p
		\\&
		+C_{p,\epsilon} K_{1}^{C\epsilon+\frac{3}{2}(\frac{1}{2}-\frac{1}{p})p} \sum_{\theta \in \mathcal{P}_{tK_{1}^{-1}}}  \|E_{\theta}f\|_{L^p(w_{B_{N}})}^p
		\\&
		+C_{p,\epsilon}K^{C\epsilon+\frac{3}{2}(\frac{1}{2}-\frac{1}{p})p}
		\sum_{\theta \in \mathcal{P}_{tK^{-1}}}\|E_{ \theta}f\|_{L^p({w}_{B_{N}})}^p
		\\ &
		+C_{p,\epsilon}K^{C\epsilon+{3}(\frac{1}{2}-\frac{1}{p})p}
		\sum_{\alpha \in \mathcal{P}_{tK^{-2}}}\|E_{ \alpha}f\|_{L^p({w}_{B_{N}})}^p
		\\ &
		+C_{p,\epsilon}K^{100p}D_{bil}(Nt,p,C_{K})^p
		\sum_{\theta \in \mathcal{P}_{N^{-1}}}\|E_{\theta}f\|_{L^p({w}_{B_{N}})}^p.
		\end{split}
		\end{displaymath}
	\end{enumerate}
\end{propo}

\begin{proof}
	We will prove only (2) because the proof of (1) and the proof of (2) are identical. We define the affine transformation associated with $\alpha$, which was defined in $(3.1)$. Let $C_{N}$ be the cylinder and take
	$g(\xi)=f(L_{\alpha}^{-1}\xi)t^{\frac{3}{2}-\frac{9}{2p}}$ as before. 
	We apply Proposition 5.2;
	\begin{displaymath}
	\begin{split}
	\|E_{\alpha}f\|_{L^p({B_{N}})}^p &= 
	\|E_{[0,1]^3}g\|_{L^p(C_{N})}^p \leq \sum_{B_{tN} \cap C_{N} \neq \phi } \|E_{[0,1]^3}g\|_{L^p(B_{tN})}^p 
	\\&
	\leq 
	C_{p,\epsilon}
	K_{2}^{C\epsilon+ \frac{3}{2}(\frac{1}{2}-\frac{1}{p})p}\sum_{\theta \in \mathcal{P}_{K_{2}^{-1}} }\|
	E_{\theta}g\|_{L^p(\sum_{B_{tN} \cap C_{N} \neq \phi}{w}_{B_{tN}})}^p
	\\&
	+C_{p,\epsilon} K_{1}^{C\epsilon+\frac{3}{2}(\frac{1}{2}-\frac{1}{p})p} \sum_{\theta \in \mathcal{P}_{K_{1}^{-1}}}  \|E_{\theta}g\|_{L^p(\sum_{B_{tN} \cap C_{N} \neq \phi}{w}_{B_{tN}})}^p
	\\&
	+C_{p,\epsilon}K^{C\epsilon+\frac{3}{2}(\frac{1}{2}-\frac{1}{p})p}
	\sum_{\theta \in \mathcal{P}_{K^{-1}}}\|E_{ \theta}g\|_{L^p(\sum_{B_{tN} \cap C_{N} \neq \phi}{w}_{B_{tN}})}^p
	\\&
	+C_{p,\epsilon}K^{C\epsilon+{3}(\frac{1}{2}-\frac{1}{p})p}
	\sum_{\alpha \in \mathcal{P}_{K^{-2}}}\|E_{ \alpha}g\|_{L^p(\sum_{B_{tN} \cap C_{N} \neq \phi}{w}_{B_{tN}})}^p
	\\ &
	+C_{p,\epsilon}K^{100p}D_{bil}(Nt,p,C_{K})^p
	\sum_{\theta \in \mathcal{P}_{(tN)^{-1}}}\|E_{\theta}g\|_{L^p(\sum_{B_{tN} \cap C_{N} \neq \phi}{w}_{B_{tN}})}^p.	
	\end{split}
	\end{displaymath}
	It suffices now to return to the original variables and use Lemma 3.1.
\end{proof}

\subsubsection*{Proof of Theorem 5.1.}
We will prove only (2) because the proof of (1) and the proof of (2) are identical.
Take $m=\frac{4\log N}{\log K_{2}}$ so that $K_{2}^{\frac{m}{8}}=N^{\frac{1}{2}}$.
We now use Proposition 4.3 repeatedly until every inverse of side length of dyadic cubes is in the interval $[N^{-\frac{1}{2}}K^{5},N^{-\frac{1}{2}}K^{10}]$ (Hence, we iterate this at most $m$ times), then apply Lemma 3.3 to decompose the cubes with such side length into the cubes with side length of $N^{-\frac{1}{2}}$. Then we have
\begin{displaymath}
\begin{split}
\|E_{[0,1]^3}f\|_{L^p({w}_{B_{N}})}^p \leq CN^{\frac{\log C_{p,\epsilon}}{\log K_{2}}}  N^{\epsilon} \sup_{1<M<N} \biggl[
(\frac{N}{M})^{\frac{3}{2}(\frac{1}{2}-\frac{1}{p})p}D_{bil}(\frac{N}{M},p,C_{K})^p \biggr]
\sum_{\theta \in \mathcal{P}_{N^{-1}}}\|E_{\theta}f\|_{L^p({w}_{B_{N}})}^p.
\end{split}
\end{displaymath}
Take $K_{2}$ large enough so that $\frac{\log C_{p,\epsilon}}{\log K_{2}} \leq \epsilon$.
This completes the proof of Theorem 5.1
\qed

\section{The equivalent formulations}
\label{sec:5}

The remaining sections contain no novelty. We will simply follow Bourgain and Demeter's argument. For a streamlined proof, we refer to \cite{BD-Astudy-2017}. 

In this section, we study well-known equivalent formulations. 


Let $S$ be a three-dimensional surface in $\R^6$.
For $N>1$ and $Q \subset [0,1]^{3}$, we define the $N^{-1}$-neighborhood of $S$ above $Q$ to be
\begin{displaymath}
\begin{split}
\mathcal{N}_{N^{-1}}(Q)=\{(\xi,
\Phi(\xi)+(t_{1},t_{2},t_{3})): \xi \in Q,\; -N^{-1} \leq t_{1},t_{2},t_{3} \leq N^{-1}	\}.
\end{split}
\end{displaymath}
Let $\pi:\mathcal{N}_{N^{-1}}([0,1]^3) \rightarrow \R^3$ be the standard projection map.
For $\nu>0$, we say that two sets $E,F \subset \mathcal{N}_{N^{-1}}([0,1]^3)$ are $\nu$-transverse if $\pi(E)$ and $\pi(F)$ are $\nu$-transverse. For a function $f$ and a measurable set $E \subset \mathbb{R}^{6}$, we denote by $f_{E}=(\hat{f}1_{E})^{\vee}$ the Fourier restriction to the set $E$. Here, the notation $\vee$ is the Fourier inverse transform.


Fix $\nu>0$. For any $2 \leq p < \infty$ and any number $N\geq1$, we denote by $\tilde{D}_{S}(N,p)$ the smallest constant such that the following decoupling holds;
$$\|f\|_{L^p({w}_{B_{N}})} \leq \tilde{D}_{S}(N,p)\bigl(\sum\limits_{Q \in \mathcal{P}_{N^{-1}} }\|f_{\mathcal{N}_{N^{-1}}(Q)}\|_{L^p({w}_{B_{N}})}^p\bigr)^{\frac{1}{p}}$$
for each $f: \mathbb{R}^{6} \rightarrow \mathbb{C}$ with Fourier support in $\mathcal{N}_{N^{-1}}([0,1]^3)$. Similarly, we denote by $\tilde{D}_{bil}(N,p,\nu)$ the smallest constant such that the following decoupling holds;
\begin{displaymath}
\begin{split}
\||f_{1}f_{2}|^{\frac{1}{2}}\|_{L^p(w_{B_{N}})}
\leq \tilde{D}_{bil}(N,p,\nu)\prod_{i=1}^{2}\bigl(\sum\limits_{Q \in \mathcal{P}_{N^{-1}} }\|(f_{i})_{\mathcal{N}_{N^{-1}}(Q)}\|_{L^p({w}_{B_{N}})}^p\bigr)^{\frac{1}{2p}}
\end{split}
\end{displaymath}
for any $f_{i}:\mathbb{R}^{6} \rightarrow \mathbb{C}$ with Fourier support in $\mathcal{N}_{N^{-1}}(Q_{i})$, where $Q_{1},Q_{2} \subset [0,1]^3$ are any $\nu$-transverse dyadic cubes.

\begin{propo}
	Let $\nu>0$ and $p\geq2$.
	For any $N \geq 1$
	\begin{gather*}
	D_{S}(N,p) \sim \tilde{D}_{S}(N,p), \\
	D_{bil}(N,p,\nu) \leq C_{p,\nu} \tilde{D}_{bil}(N,p,\nu).
	\end{gather*}
\end{propo}

The proof of Proposition 6.1 is identical to that of Theorem 5.1 in \cite{BD-Astudy-2017}.

\begin{proof}
	We may assume that the cube $B_{N}$ in the definition of $D_{S}(N,p)$ is $[0,N]^{6}$. Let $g:\R^3 \rightarrow \C$ be a function.  Define a function $f$ to be
	\begin{displaymath}
	\hat{f}(\xi,\Phi(\xi)+(\tau_{1},\tau_{2},\tau_{3}))=g(\xi)\prod_{i=1}^{3}1_{[0,N^{-1}/10]}(\tau_{i})	\end{displaymath}
	Note that 
	$$f(x_{1},\ldots,x_{6})=E_{[0,1]^3}g(x_{1},\ldots,x_{6})
	\prod_{i=1}^{3}
	\int_{0}^{N^{-1}/10}e(tx_{3+i})\,dt,$$
	and 
	$$f_{\mathcal{N}_{N^{-1}}(Q)}(x)=E_{Q}g(x)\prod_{i=1}^{3}\int_{0}^{N^{-1}/10}e(t x_{i+d})\,dt.$$
	Note also that
	$|\int_{0}^{N^{-1}/10} e(tx_{3+i})\,dt | \sim N^{-1}$ if $x_{3+i} \in [0,N]$.
	These give
	\begin{displaymath}
	\begin{split}
	\|E_{[0,1]^3}g\|_{L^p(B_{N})} &\lesssim N^{3}\|f\|_{L^p(w_{B_{N}})} 
	\lesssim N^{3}
	\tilde{D}_{S}(N,p)
	\bigl(\sum_{Q \in \mathcal{P}_{N^{-1}}}\|f_{\mathcal{N}_{N^{-1}}(Q)}\|_{L^p(w_{B_{N}})}^p		\bigr)^{\frac{1}{p}}
	\\&
	\lesssim \tilde{D}_{S}(N,p) \bigl(
	\sum_{Q \in \mathcal{P}_{N^{-1}}}
	\|E_{Q}g\|_{L^p(w_{B_{N}})}^p
	\bigr)^{\frac{1}{p}}.
	\end{split}
	\end{displaymath}
	Now, Lemma 3.1 gives $D_{S}(N,p) \lesssim \tilde{D}_{S}(N,p)$. Similarly, one can get $D_{bil}(N,p,\nu) \lesssim \tilde{D}_{bil}(N,p,\nu)$.
	
	Now, we will show that $\tilde{D}_{S}(N,p) \lesssim D_{S}(N,p)$. By a change of variables,
	\begin{displaymath}
	\begin{split}
	f(x_{1},\ldots,x_{6}) &=\int_{\mathcal{N}_{N^{-1}}([0,1]^3)}\hat{f}(\xi,\tau)e((\xi,\tau)\cdot x)\,d\xi d\tau
	\\& = \sum_{Q \in \mathcal{P}_{N^{-1}}}
	\int_{Q \times [-N^{-1},N^{-1}]^{3}} \hat{f}(\xi,\Phi(\xi)+\tau)e((\xi,\Phi(\xi))\cdot x)e(\tau \cdot (x_{4},x_{5},x_{6}))\,d\xi d\tau.
	\end{split}
	\end{displaymath}
	We will deal with the term $e(\tau \cdot (x_{4},x_{5},x_{6}))$ by using the Taylor expansion
	\begin{displaymath}
	\begin{split}
	e(\tau \cdot (x_{4},x_{5},x_{6}))=\prod_{k=1}^{3}
	\sum_{j=0}^{\infty} \frac{1}{j!}(\frac{2ix_{3+k}}{N})^j(\frac{N\tau_{i}}{2})^j.
	\end{split}
	\end{displaymath}
	By putting this, for $x \in B_{N}$ we have
	\begin{displaymath}
	\begin{split}
	|f(x)| \leq \sum_{j_{1},j_{2},j_{3}}\frac{2^{j_{1}}2^{j_{2}} 2^{j_{3}}}{j_{1}!j_{2}! j_{3}!}|\sum_{Q \in \mathcal{P}_{N^{-1}}}
	E_{Q}g_{j_{1},j_{2},j_{3}}(x)|,
	\end{split}
	\end{displaymath}
	where $$g_{j_{1},j_{2},j_{3}}(\xi)=\int_{[-N^{-1},N^{-1}]^3}\hat{f}(\xi,\Phi(\xi)+\tau)\prod_{i=1}^{3}(\frac{N\tau_{i}}{2})^{j_{i}}\, d\tau_{1} d\tau_{2} d\tau_{3}.$$	
	From the definition of ${D}_{S}(N,p)$, we have
	\begin{displaymath}
	\begin{split}
	\|f\|_{L^p({B_{N}})} \lesssim {D}_{S}(N,p)
	\sum_{j_{1},j_{2},j_{3}}\frac{2^{j_{1}}2^{j_{2}} 2^{j_{3}}}{j_{1}!j_{2} j_{3}!}
	\bigl(		\sum_{Q \in \mathcal{P}_{N^{-1}}}\|E_{Q}g_{j_{1},j_{2},j_{3}}\|_{L^p(w_{B_{N}})}^p
	\bigr)^{\frac{1}{p}}.
	\end{split}
	\end{displaymath}
	Fix $Q=c+[0,N^{-\frac{1}{2}}]=(c_{1},c_{2},c_{3})+[0,N^{-\frac{1}{2}}]^3$. By Lemma 3.1, the inequality
	\begin{displaymath}
	\begin{split}
	\|E_{Q}g_{j_{1},j_{2},j_{3}}\|_{L^p(w_{B_{N}})} \lesssim \|f_{\mathcal{N}_{N^{-1}}(Q)}\|_{L^p(w_{B_{N}})},
	\end{split}
	\end{displaymath}
	which is uniform over $j_{1},j_{2},j_{3}$, implies the desired results, and this follows from
	\begin{equation}
	\begin{split}
	\|E_{Q}g_{j_{1},j_{2},j_{3}}\|_{L^p({B_{N}})} \lesssim \|f_{\mathcal{N}_{N^{-1}}(Q)}\|_{L^p(w_{B_{N}})}.
	\end{split}
	\end{equation}
	We take a Schwartz function $M_{j}(t)$ which agrees with $t^{j}$ on $[-\frac{1}{2},\frac{1}{2}]$ and satisfies the derivative bound
	\begin{displaymath}
	\begin{split}
	\|\frac{d^s}{dt^{s}}M_{j}\|_{L^{\infty}(\mathbb{R})} \lesssim_{s} 1,
	\end{split}
	\end{displaymath}
	uniformly over $j \geq 1$, for each $s \geq 0$.
	The following equality gives the relation between $E_{Q}g_{j_{1},j_{2},j_{3}}$ and $ f_{\mathcal{N}_{N^{-1}}(Q)}$:
	\begin{equation}
	\begin{split}
	E_{Q}g_{j_{1},j_{2},j_{3}}(x) &
	=\int_{\mathcal{N}_{N^{-1}}(Q)} \hat{f}(\xi,\tau)m_{j_{1},j_{2},j_{3}}(\xi,\tau)e((\xi,\tau) \cdot x)\, d\xi d\tau
	\\&= (f_{\mathcal{N}_{N^{-1}}(Q)}* ({m_{j_{1},j_{2},j_{3}}})^{\vee})(x)
	.
	\end{split}
	\end{equation}
	Here, $m_{j_{1},j_{2},j_{m}}$ is defined by 
	\begin{gather*}
		m_{j_{1},j_{2},j_{3}}(\xi,\tau) 
		=
		e((\Phi(\xi)-\tau)\cdot (x_{4},x_{5},x_{6}))
		\prod_{i=1}^{3}M_{j_{i}}(\frac{N(\tau_{i}-\Phi_{i}(\xi))}{2})
		\prod_{i=1}^{3}H(N^{\frac{1}{2}}(\xi_{i}-c_{i})),
	\end{gather*}
	and $H(\xi_{i})$ is a Schwartz function equal to $1$ on $[0,1]$ and $0$ on $(-\infty,-1] \cup [2,\infty)$.
	The above equality immediately follows from a change of variables. Now, we will estimate $({m_{j_{1},j_{2},j_{3}}})^{\vee}(y)$. 
	A change of variables gives
	\begin{displaymath}
	\begin{split}
	|({m_{j_{1},j_{2},j_{3}}})^{\vee}(y)| &
	=|\prod_{i=1}^{3}\frac{2}{N}\widehat{M}_{j_{i}}
	(\frac{2(x_{3+i}-y_{3+i})}{N})||
	\int e((\xi,\Phi(\xi+c)) \cdot y) \prod_{i=1}^{3}H({N^{\frac{1}{2}}\xi_{i}})\,d\xi|
	.
	\end{split}
	\end{displaymath}
	Since $\Phi_{i}$ are quadratic polynomials, we can write $\Phi(\xi+c)=\Phi(\xi)+\Phi(c)+\xi A$ for some 3 by 3 matrix $A$ (with transpose $A^{T}$). Hence, we can write the second term as
	\begin{displaymath}
	\begin{split}
	\frac{1}{N^{{3}/{2}}} |\int e((\xi,\Phi(\xi))\cdot (\frac{(y_{1},y_{2},y_{3})+A^T(y_{4},y_{5},y_{6})}{N^{1/2}},\frac{(y_{4},y_{5},y_{6})}{N})) \prod_{i=1}^{3}H(\xi_{i})\, d\xi|.
	\end{split}
	\end{displaymath}
	By integration by parts and the construction of the function $M_{j}$, for $x\in B_{N}$
	\begin{displaymath}
	\begin{split}
	|({m_{j_{1},j_{2},j_{3}}})^{\vee}(y)| \lesssim
	\frac{N^{-3}}{1+|\frac{(y_{4}-x_{4},y_{5}-x_{5},y_{6}-x_{6})}{N}|^{500}} \frac{N^{-\frac{3}{2}}}{1+|\frac{(y_{1},y_{2},y_{3})+A^T (y_{4},y_{5},y_{6})}{N^{1/2}}|^{500}}.
	\end{split}
	\end{displaymath}
	Now, we are ready to obtain $(6.1)$. By $(6.2)$, Young's inequality and the above inequality
	\begin{gather*}
		\|E_{Q}g_{j_{1},j_{2},j_{3}}\|_{L^p(B_{N})}^p 
		\lesssim  \|
		f_{\mathcal{N}_{N^{-1}}(Q)}*({m_{j_{1},j_{2},j_{3}}})^{\vee}
		\|_{L^p(B_{N})}^p
		\lesssim \|f_{\mathcal{N}_{N^{-1}}(Q)}  \|_{L^p(w_{B_{N}})}^p.
	\end{gather*}
	This completes the proof of Proposition 6.1.
\end{proof}

Take a collection of non-negative smooth functions $\{\chi(\, \cdot+k) \}_{k \in \mathbb{N}}$ such that  $\chi(\xi)=1$ if $\xi \in [-1,1]$ and $\chi(\xi)=0$ if $\xi \in [-2,2]^{3}$.
For each cube $Q=(i_{1},i_{2},i_{3})+[0,N^{-\frac{1}{2}}]^{3} \in \mathcal{P}_{N^{-1}}$, we define a function $\Xi_{Q}$ to be
\begin{displaymath}
\begin{split}
\widehat{\Xi_{Q}}(\xi_{1},\ldots,\xi_{6})=\prod_{k=1}^{3}
\frac{\chi(N^{\frac{1}{2}}(\xi_{k}-i_{k}))}{\sum_{m \in \mathbb{Z}^{3} }\chi(N^{\frac{1}{2}}\xi_{k}-m)}
\prod_{k=1}^{3}
\chi(N({\xi_{3+k}-\Phi_{k}(\xi_{1},\xi_{2},\xi_{3})})).
\end{split}
\end{displaymath}
Note that $\|\Xi_{Q}\|_{L^1} \sim 1$, $\mathrm{supp}(\widehat{\Xi_{Q}}) \subset \mathcal{N}_{5N^{-1}}(5Q)$ and $\sum_{Q \in \mathcal{P}_{N^{-1}} }\widehat{\Xi_{Q}}(\xi) =1$ for all $\xi \in \mathcal{N}_{N^{-1}}([0,1]^3)$.

We denote by $D^{(1)}_{S}(N,p), D^{(2)}_{S}(N,p)$ the smallest constants such that the following decoupling holds;
\begin{gather*}
	\|f\|_{L^p({w}_{B_{N}})} \leq {D}^{(1)}_{S}(N,p)\bigl(\sum\limits_{Q \in \mathcal{P}_{N^{-1}} }\|f* \Xi_{Q} \|_{L^p({w}_{B_{N}})}^p\bigr)^{\frac{1}{p}}
	\\
	\|f\|_{L^p(\mathbb{R}^{6})} \leq {D}^{(2)}_{S}(N,p)\bigl(\sum\limits_{Q \in \mathcal{P}_{N^{-1}} }\|f*\Xi_{Q}\|_{L^p(\mathbb{R}^{6})}^p\bigr)^{\frac{1}{p}}
\end{gather*}
for any $f:\mathbb{R}^{6} \rightarrow \mathbb{C}$ with Fourier support in $\mathcal{N}_{N^{-1}}([0,1]^3)$. 

By using the fact that $\{\widehat{\Xi_{Q}}\}$ forms a partition of unity on $\mathcal{N}_{N^{-1}}([0,1]^3)$ and has a finitely overlapping property, one can prove the following proposition.

\begin{propo}
	Let $p \geq 2$.
	For any $N \geq 1$
	\begin{gather*}
	 \tilde{D}_{S}(N,p) \sim {D}^{(1)}_{S}(N,p) \sim {D}^{(2)}_{S}(N,p).
	\end{gather*}
\end{propo}
The same observation applies to the family of constants related to $\tilde{D}_{bil}(N,p,\nu)$.

We need parabolic rescaling for the equivalent formulations. The proof of Proposition 6.3 is identical to that of Proposition 3.2. Hence, we will omit the detail.

\begin{propo}[Parabolic rescaling]
	Suppose that two numbers $N,\sigma$ with $0< N^{-1} \leq \sigma $ are dyadic numbers, and let $\tau=a+[0,\sigma^{\frac{1}{2}}]^3 \in \mathcal{P}_{\sigma}$. Then for each $f: \mathbb{R}^{6} \rightarrow \mathbb{C}$ with Fourier support in $\mathcal{N}_{\sigma}(\tau)$, we have
	\begin{displaymath}
	\begin{split}
	\|f\|_{L^p(\mathbb{R}^{6})} \lesssim 
	{{D}}_{S}({N\sigma},p)
	\bigl( \sum\limits_{\substack{\theta \in \mathcal{P}_{N^{-1}}, \, \theta \subset \tau}}\|f*\Xi_{\theta}\|_{L^p(\mathbb{R}^{6})}^p	\bigr)^{\frac{1}{p}}.
	\end{split}
	\end{displaymath}
\end{propo}

\section{The wave packet decomposition}
\label{sec:7}

In this section, we will obtain the wave packet decomposition, which will be used to prove Proposition 8.3. The proof of the wave packet decomposition is well known. We will follow the proof in \cite{ GSS-Improvements-2008, GS-On-2009}.

For each rectangle $R$, we denote by $a_{R}$ an affine map taking $[0,1]^{6}$ to the rectangle $R$.
We take a Schwartz function $\phi$ such that the function is strictly positive in $B_{2}(0)$, the Fourier support is in $B_{C}(0)$ and $\sum_{n \in \mathbb{Z}^{6}}\phi(\cdot + n)^2=1$ for some $C>0$. Let $\phi_{R}=\phi \circ a_{R}^{-1}$.


\begin{deff}
	Let $N \geq 1$. Let $\theta=c+[0,N^{-\frac{1}{2}}]^3$, $c \in \mathbb{R}^{3}$. We take a rectangular box $R_{\theta}$ of dimensions $C(N^{-\frac{1}{2}} \times N^{-\frac{1}{2}} \times N^{-\frac{1}{2}} \times N^{-1} \times N^{-1} \times N^{-1})$ such that 
	\begin{enumerate}
		\item $\mathcal{N}_{N^{-1}}(\theta) \subset R_{\theta}$
		\item the short directions are parallel to the subspace spanned by $m_{1}(c),m_{2}(c),m_{3}(c)$
	\end{enumerate}
	for some constant $C$ independent of $N$ and the choice of $\theta$.
	We denote the dual set of $R_{\theta}$ by $R_{\theta}^*$, and we write $R_{\theta}^*\| \theta$ if the above conditions are satisfied.
\end{deff}

\begin{lem}[The wave packet decomposition]
	Let $N \geq 1$ and $\epsilon >0$. Let $Q$ be a cube with a side length of $2N$ in $\mathbb{R}^{6}$. Let $f$ be a function with $\mathrm{supp}\hat{f} \subset \mathcal{N}_{N^{-1}}([0,1]^3)$. Assume that 
	$$\sup_{\theta \in \mathcal{P}_{N^{-1}}} \|f*\Xi_{\theta}\|_{L^{\infty}(\mathbb{R}^{6})} \leq A$$
	for some number $A$.
	Then we can decompose $f$ into
	\begin{equation}
	\begin{split}
	f(x)=\sum_{AN^{-30} \lesssim	 2^m \lesssim A} \sum_{j=1}^{O(\log N)}f^{[j,m]}(x) + g(x), \;\;\;\;\; x \in \mathbb{R}^{6},
	\end{split}
	\end{equation}
	such that $f^{[j,m]}$ and $g$ satisfy the following:
	\begin{enumerate}
		\item The function $g$ is an essentially error function. More precisely,
		$$\|g\|_{L^{\infty}(Q)} \lesssim_{\epsilon} N^{-24}A.$$
		\item For every $2 \leq p < \infty$, $j$ and $m$, we have
		\begin{equation}
		\begin{split}
		&\|f^{[j,m]}\|_{L^2(\mathbb{R}^{6})}^{2}
		(\sum_{\theta \in \mathcal{P}_{N^{-1}} }\|f^{[j,m]}*\Xi_{\theta}\|_{L^{\infty}(\mathbb{R}^{6})}^2)^{\frac{p-2}{2}}
		\lesssim (\sum_{\theta \in \mathcal{P}_{N^{-1}}}\|f*\Xi_{\theta}\|_{L^p(\mathbb{R}^{6})}^2)^{\frac{p}{2}}.
		\end{split}
		\end{equation}
	\end{enumerate}
\end{lem}
By using the above inequality, we can recover the original function $f$ from the packets.

\begin{proof}
	We decompose $f$ by dividing a frequency space; $f=\sum_{\theta \in \mathcal{P}_{N^{-1}}}f*\Xi_{\theta}$. Next, we decompose each $f*\Xi_{\theta}$ by splitting a physical space; $f*\Xi_{\theta}=\sum_{\pi \in L : \pi \| \theta} (f*\Xi_{\theta})\phi_{\pi}^2$, where $L=\{ \pi \}$ is a tiling of $\mathbb{R}^{6}$. We define
	$\mathcal{L}_{\theta,Q}=
	\{	\pi \in L \, : \, \pi \| \theta,\;
	\pi \cap 2N^{\epsilon}Q \neq \phi
	\}$. Note that $|\mathcal{L}_{\theta,Q}| \lesssim N^{3}$.

	Now, we exclude error terms. Define the error function $g$ to be
	\begin{displaymath}
	\begin{split}
	g&=\sum_{\theta \in \mathcal{P}_{N^{-1}} }
	\biggl( \sum_{\substack{
			\pi \in L : \pi \notin \mathcal{L}_{\theta,Q}}}(f*\Xi_{\theta})\phi_{\pi}^2+
	\sum_{\substack{\pi \in \mathcal{L}_{\theta,Q}\\\|(f*\Xi_{\theta})\phi_{\pi}\|_{L^{\infty}(\mathbb{R}^{6})} \leq AN^{-30}}}(f*\Xi_{\theta})\phi_{\pi}^2 \biggr).
	\end{split}
	\end{displaymath}
	To show the first property in Lemma 7.2, we use a Schwartz tail of the function $\phi_{\pi}$;
	\begin{displaymath}
	\begin{split}
		\|g\|_{L^{\infty}(Q)}    
		&\leq  
		\sum_{\theta} \biggl( \sum\limits_{\substack{\pi \in L :\\ \pi \notin \mathcal{L}_{\theta,Q}}}
		\|(f*\Xi_{\theta})\phi_{\pi}^2\|_{L^{\infty}(Q)} 
		+\sum\limits_{\substack{\pi \in \mathcal{L}_{\theta,Q} \\
				\|(f*\Xi_{\theta})\phi_{\pi}\|_{L^{\infty}(\mathbb{R}^{6})} \leq AN^{-30}}}
		\|(f*\Xi_{\theta})\phi_{\pi}^2\|_{L^{\infty}(Q)} \biggr)
		\\ &
		\lesssim_{\epsilon} N^{-300}\sup_{\theta}\|f*\Xi_{\theta}\|_{L^{\infty}(\mathbb{R}^{6})}+N^{6}AN^{-30}
		\lesssim N^{-24}A.
	\end{split}
	\end{displaymath}
	Hence, the first property follows.

	The main term can be written as
	\begin{displaymath}
	\begin{split}
	f-g=\sum_{\theta \in \mathcal{P}_{N^{-1}} }\sum\limits_{\substack{\pi \in \mathcal{L}_{\theta,Q} \\ \|(f*\Xi_{\theta})\phi_{\pi}\|_{L^{\infty}(\mathbb{R}^{6})} >AN^{-30} }} (f*\Xi_{\theta})\phi_{\pi}^2.
	\end{split}
	\end{displaymath}
	For each cube $\theta \in \mathcal{P}_{N^{-1}}$ and $m \in \mathbb{Z}$, we define
	\begin{displaymath}
	\begin{split}
	\mathcal{L}_{\theta,Q}^{m}=
	\{\pi \in \mathcal{L}_{\theta,Q}: 2^m < \|(f*\Xi_{\theta})\phi_{\pi}\|_{L^{\infty}(\mathbb{R}^{6})} \leq 2^{m+1}	\}.
	\end{split}
	\end{displaymath}
	Since
	$\| (f*\Xi_{\theta})\phi_{\pi} \|_{L^{\infty}(\mathbb{R}^{6})} \leq \| f*\Xi_{\theta}\|_{L^{\infty}(\mathbb{R}^{6}) } \leq A$, we can see $2^{m} \leq A$ if the set $\mathcal{L}_{\theta,Q}^{m}$ is non-empty.
	Next, for each $j \in \mathbb{Z}$ we define
	\begin{displaymath}
	\begin{split}
	E_{j,m}=\{\theta \in \mathcal{P}_{N^{-1}} : 2^{j} < |\mathcal{L}_{\theta,Q}^m| \leq 2^{j+1}	\}.
	\end{split}
	\end{displaymath}
	We can also see that $2^{j} \leq CN^{{3}}$ for some $C>0$ if the $E_{j,m}$ is non-empty, so the set $E_{j,m}$ is non-empty only for $j \lesssim \log{N}$.
	Now, we define the functions associated with $(j,m)$ by
	\begin{displaymath}
	\begin{split}
	f^{[j,m]}=\sum_{\theta \in E_{j,m}}
	\sum_{\pi \in \mathcal{L}_{\theta,Q}^{m}}(f*\Xi_{\theta})\phi_{\pi}^2.
	\end{split}
	\end{displaymath}
	We write $\mathcal{L}^{j,m} = \cup_{\theta \in E_{j,m}	} \mathcal{L}_{\theta,Q}^{m}$. Note that the equality $(7.1)$ holds and $|\mathcal{L}^{j,m}| = |E_{j,m}||\mathcal{L}_{\theta,Q}^{m}|$.
	
	We will show the inequality $(7.2)$. Observe that
	\begin{displaymath}
	\begin{split}
	\|f^{[j,m]}\|_{L^2(\mathbb{R}^{6})}^2 \lesssim 2^{2m} N^{\frac{9}{2}}|\mathcal{L}^{j,m}| \;\;\; \mathrm{and} \;\;\;
	\sum_{\theta}\|f^{[j,m]}*\Xi_{\theta}\|_{L^\infty(\mathbb{R}^{6})}^2 \lesssim 2^{2m}|E_{j,m}|.
	\end{split}
	\end{displaymath}
	These inequalities follow from an orthogonality property. The second property in Lemma 7.2 follows from
	\begin{displaymath}
	\begin{split}
	2^{mp}N^{\frac{9}{2}}|\mathcal{L}^{j,m}||E_{j,m}|^{\frac{p-2}{2}} \lesssim (\sum_{\theta \in \mathcal{P}_{N^{-1}}}\|f*\Xi_{\theta}\|_{L^p(\mathbb{R}^{6})}^2)^{\frac{p}{2}}.
	\end{split}
	\end{displaymath}
	Note that
	$|\mathcal{L}_{\theta_{1},Q}^{m}| \sim |\mathcal{L}_{\theta_{2},Q}^{m}| \sim 2^j$ for any $\theta_{1}$ and $\theta_{2}$ in $E_{j,m}$ and $|\mathcal{L}^{j,m}| \lesssim 2^j|E_{j,m}|$. 
	This implies that for any $\theta \in E_{j,m}$
	\begin{displaymath}
	\begin{split}
	2^{mp}N^{\frac{9}{2}}|\mathcal{L}^{j,m}||E_{j,m}|^{\frac{p-2}{2}}
	& \lesssim 2^{mp}N^{\frac{9}{2}}2^{j}|E_{j,m}|^{\frac{p}{2}} \lesssim 
	|E_{j,m}|^{\frac{p}{2}} 
	\sum_{\pi \in \mathcal{L}_{\theta,Q}^{m}} N^{\frac{9}{2}}2^{mp} \\ &
	\lesssim |E_{j,m}|^{\frac{p}{2}}
	\sum_{\pi \in \mathcal{L}_{\theta,Q}^{m}} \|(f*\Xi_{\theta})\phi_{\pi}\|_{L^{\infty}(\mathbb{R}^{6})}^pN^{\frac{9}{2}}
	\\ &
	\lesssim |E_{j,m}|^{\frac{p}{2}}
	\sum_{\pi \in \mathcal{L}_{\theta,Q}^{m}} \|(f*\Xi_{\theta})\phi_{\pi}\|_{L^p(\mathbb{R}^{6})}^p \lesssim 
	|E_{j,m}|^{\frac{p}{2}} \|f*\Xi_{\theta}\|_{L^p(\mathbb{R}^{6})}^p
	\end{split}
	\end{displaymath}
	by Bernstein's inequality and $\sum_{\pi}|\phi_{\pi}|^p \lesssim 1$.
	Raising to the power $\frac{2}{p}$ and summing over all $\theta \in E_{j,m}$ lead to the desired inequality.
\end{proof}

\section{Proof of Theorem 1.1}
\label{sec:7}

We will follow the multiscale argument in \cite{G-Decoupling-2014}. The only difference between our proof and the proof in \cite{G-Decoupling-2014} is that we obtain the $l^p$ decoupling instead of the $l^2$ decoupling, but this will not make any trouble.

For simplicity, we write
\begin{gather*}
\|f\|_{L^{p,N}(\mathbb{R}^{6})}=
(\sum\limits_{\substack{\theta \in \mathcal{P}_{N^{-1}}}}\|f*\Xi_{\theta}\|_{L^p(\mathbb{R}^{6})}^p)^{\frac{1}{p}},
\\
\|f\|_{L^{p,N}(w_{B_{N}})}=
(\sum\limits_{\substack{\theta \in \mathcal{P}_{N^{-1}}}}\|f*\Xi_{N^{-1}}\|_{L^p(w_{B_{N}})}^p)^{\frac{1}{p}}
\end{gather*}
for any $N \geq1$ and $ 1 \leq p < \infty$. 

By Theorem 5.1, we can assume that $\gamma_{lin} \leq \gamma_{bil}$. Moreover, Theorem 1.1 follows from
\begin{thm}
	For $s \geq 6$, we have
	\begin{displaymath}
	\begin{split}
	\gamma_{bil}(s) \leq \frac{3}{2}-\frac{6}{s}.
	\end{split}
	\end{displaymath}
\end{thm}

We first prove the following inequality. The proof of this is very similar to that of Proposition 2.1 in \cite{BCT-On-2006}, that of Proposition 4.7 in \cite{B-Aspects-2014} and that of Lemma 4.4 in \cite{TVV-A-1998}.

\begin{propo}
	Fix $N \geq 1$ and $\nu>0$. 
	If $\mathrm{supp}\hat{f}_{i} \subset \mathcal{N}_{N^{-1}}([0,1]^3)$ and the Fourier supports of $f_{i}$ are $\nu$-transverse, then for each $p \geq 4$
	\begin{displaymath}
	\int_{B_{N}} \prod_{i=1}^{2}\|f_{i}\|_{L^2({B_{N^{1/2}}}(x))}^{\frac{p}{2}}
	\, dx
	\lesssim_{\nu}{N^{3(1-\frac{p}{4})}}
	\prod_{i=1}^{2}\|f_{i}\|_{L^2(\mathbb{R}^{6})}^{\frac{p}{2}}.
	\end{displaymath}
\end{propo}

\begin{proof}
	Let $\eta$ be the function defined at the beginning of the proof of Proposition 5.2. Since the collection $\{\widehat{(f_{i}*\Xi_{\theta})}*\widehat{\eta_{B_{N^{1/2}}(x)}}\}_{\theta \in \mathcal{P}_{N^{-1}} }$ is finitely overlapping, we have
	\begin{displaymath}
	\begin{split}
	\int_{B_{N}}\prod_{i=1}^{2}\|f_{i}\|_{L^2(B_{N^{1/2}}(x))}^{\frac{p}{2}}\,dx \lesssim 	\int_{B_{N}}\prod_{i=1}^{2}(\sum_{\theta
		\in \mathcal{P}_{N^{-1}}
	}\|f_{i}*\Xi_{\theta}\|_{L^2(\eta_{B_{N^{1/2}}(x))}}^{2})^{\frac{p}{4}}\,dx.
	\end{split}
	\end{displaymath}
	Hence, it suffices to prove that
	\begin{displaymath}
	\begin{split}
	\int_{B_{N}}\prod_{i=1}^{2}(\sum_{\theta\in \mathcal{P}_{N^{-1}} }\|f_{i}*\Xi_{\theta}\|_{L^2(\eta_{B_{N^{1/2}}(x))}}^{2}\,)^{\frac{p}{4}}dx \lesssim_{\nu} N^{3(1-\frac{p}{4})}
	\prod_{i=1}^{2}\|f_{i}\|_{L^2(\mathbb{R}^{6})}^{\frac{p}{2}}.
	\end{split}
	\end{displaymath}
	Moreover, by a Schwartz tail of $\eta$, it suffices to show that for any $c_{1},c_{2} \in \mathbb{R}^{6}$
	\begin{displaymath}
	\begin{split}
	\int_{B_{N}}\prod_{i=1}^{2}(\sum_{\theta\in \mathcal{P}_{N^{-1}} }\|f_{i}*\Xi_{\theta}\|_{L^2({B_{cN^{1/2}}(x+c_{i})}}^{2}\,)^{\frac{p}{4}}dx \lesssim_{\nu} N^{3(1-\frac{p}{4})}
	\prod_{i=1}^{2}\|f_{i}\|_{L^2(\mathbb{R}^{6})}^{\frac{p}{2}}
	\end{split}
	\end{displaymath}
	for some sufficiently small $c>0$.
	For each $\theta \in \mathcal{P}_{N^{-1}}$, we take a rectangular box $\theta_{0}^*$ so that $\theta_{0}^* \| \theta$ and $\mathcal{N}_{5N^{-1}}(5\theta) \subset \theta_{0}$. 
	Note that $\eta_{\theta_{0}} = 1$ on $\theta_{0}$ and 
	\begin{displaymath}
	\begin{split}
	|{\eta_{\theta_{0}}}^{\vee}(x+y)| \lesssim N^{-\frac{9}{2}}\chi_{\theta_{0}^{*}}(-x)
	\end{split}
	\end{displaymath}
	for all $x,y$ with $y \in [0,cN^{\frac{1}{2}}]^{6}$. Define $({\tilde{f}_{i,\theta}})^{\vee}(\xi)=e(-\xi \cdot c_{i})\widehat{({f}_{i}*\Xi_{\theta})}(\xi)/\eta_{\theta_{0}}(\xi)$. By Cauchy-Schwartz's inequality and the above inequality, we have
	\begin{displaymath}
	\begin{split}
	|(f_{i}*\Xi_{\theta})(x+y)|^2 &  \lesssim
	(\int_{\mathbb{R}^{6}} |\tilde{f}_{i,\theta}(z-x-y+c_{i})({\eta_{\theta_{0}}})^{\vee}(z)|\,dz)^2 
	\\ & \lesssim \int_{\mathbb{R}^{6}} |\tilde{f}_{i,\theta}(z-x+c_{i})|^2|{(\eta_{\theta_{0}})}^{\vee}(z+y)|\,dz \\&
	\lesssim N^{-\frac{9}{2}}(|\tilde{f}_{i,\theta}|^2*\chi_{\theta_{0}^{*}})(-x+c_{i})
	\end{split}
	\end{displaymath}
	for any $y \in [0,cN^{\frac{1}{2}}]^{6}$ and $x \in \mathbb{R}^{6}$. Integrating this in $y$ variable, we conclude
	\begin{displaymath}
	\begin{split}
	\|f_{i}*\Xi_{\theta}\|_{L^2(B_{cN^{1/2}}(x+c_{i}))}^2 \lesssim N^{-\frac{3}{2}}(|\tilde{f}_{i,
		\theta}|^2*\chi_{\theta_{0}^{*}})(-x).
	\end{split}
	\end{displaymath}
	Hence, we have
	\begin{gather*}
		\int_{B_{N}}\prod_{i=1}^{2}(\sum_{\theta \in \mathcal{P}_{N^{-1}}}\|f_{i}*\Xi_{\theta}\|_{L^2(B_{cN^{{1}/{2}}}(x+c_{i}))}^2)^{\frac{p}{4}}\,dx 
		\lesssim 
		N^{-\frac{3p}{4}}	\int_{B_{N}}\prod_{i=1}^{2}(\sum_{\theta \in \mathcal{P}_{N^{-1}}}|\tilde{f}_{i,\theta}|^2*\chi_{\theta_{0}^{*}}(x))^{\frac{p}{4}}\,dx.
	\end{gather*}
	To apply Corollary 2.4, we use the change of variables: $y=N^{-\frac{1}{2}}x$. Then the above term is bounded by
	\begin{displaymath}
	\begin{split}
	& \lesssim N^{3(1-\frac{p}{4})}\int_{\mathbb{R}^{6}}
	(\prod_{i=1}^{2}\sum_{\theta}|\tilde{f}_{i,\theta}|^2*\chi_{\theta_{0}^*}(N^{\frac{1}{2}}y))^{\frac{p}{4}}\,dy
	\\ & \lesssim_{\nu}N^{3(1-\frac{p}{4})}
	(\prod_{i=1}^{2}
	\int_{\mathbb{R}^{6}}
	\sum_{\theta}|\tilde{f}_{i,\theta}|^2(N^{\frac{1}{2}}z)N^{3}\,dz )^{\frac{p}{4}}
	\lesssim N^{3(1-\frac{p}{4})}
	\prod_{i=1}^{2}
	\|f_{i}\|_{L^2(\mathbb{R}^{6})}^{\frac{p}{2}}.
	\end{split}
	\end{displaymath}	
	The last inequality follows from Plancherel's theorem and the pointwise comparability of $|({\tilde{f}_{i,\theta}})^{\vee}|$ and $|\widehat{f_{i}*\Xi_{\theta}}|$.
\end{proof}

Since we are interested in the decoupling, we have to change $L^2$ norm on the right hand side in Proposition 8.2 into $L^{p,N}$ norm. However, for $p \geq 6$, the exponent of $N$ is too large to obtain the desired results if we simply apply H\"{o}lder's inequality to $L^{2,N}$ norm to obtain $L^{p,N}$ norm. As a compromise, we use a half and half mix of $L^2$ norm and $L^{p,N}$ norm.

\begin{propo} Fix $N \geq 1$, $\nu>0$ and $s \geq 6$. If $\mathrm{supp}\widehat{f_{i}} \subset \mathcal{N}_{N^{-1}}([0,1]^3)$ and the Fourier supports of $f_{i}$ are $\nu$-transverse, then for each $\epsilon>0$, we have
	\begin{gather*}
		\int_{B_{N}} \prod_{i=1}^{2}\|f_{i}\|_{L^2({w}_{B_{N^{{1}/{2}}}}(x))}^{\frac{s}{2}}
		\, dx		 
		\lesssim_{\epsilon,\nu} 
		 N^{\epsilon+\frac{3s}{4}}
		\prod_{i=1}^{2}\|f_{i}\|_{L^2({w}_{B_{N}+a_{i,N} })}^{\frac{s}{4}}\prod_{i=1}^{2}\|f_{i}\|_{L^{s,N}({w}_{B_{N}+a_{i,N} })}^{\frac{s}{4}}
	\end{gather*}
	for some point $a_{i,N}$ depending on a choice of $f_{i}$ but not a center of $B_{N}$.
\end{propo}

\begin{proof}
	Take $p=\frac{2+s}{2} \geq 4$.
	We will prove an unweighted  inequality first;
	\begin{gather*}
		\int_{B_{N}} \prod_{i=1}^{2}\|f_{i}\|_{L^2({B_{N^{{1}/{2}}}}(x))}^{\frac{s}{2}}
		\, dx		 
		\lesssim_{\epsilon,\nu} 
		N^{\epsilon+\frac{3s}{4}}
		\prod_{i=1}^{2}\|f_{i}\|_{L^2(\mathbb{R}^{6})}^{\frac{s}{4}}\prod_{i=1}^{2}\|f_{i}\|_{L^{s,N}(\mathbb{R}^{6})}^{\frac{s}{4}}.
	\end{gather*}
	We apply Lemma 7.2 with $Q=B_{2N}$ to the functions $f_{i}$. Since the error functions $g_{i}$ are much tiny compared to $f_{i}$, we can ignore these functions. For convenience, we reorder indices $[j,m]$ in Lemma 7.2 so that we can write  $f_{i}=\sum_{l=1}^{O(N^{\epsilon})}f_{i,l}+g$. Then we have
	\begin{displaymath}
	\begin{split}
	\int_{B_{N}} \prod_{i=1}^{2}\|f_{i}\|_{L^2(B_{N^{1/2}}(x))}^{\frac{s}{2}}
	\, dx		&
	\lesssim_{\epsilon}N^{\epsilon} \max_{l_{1},l_{2}} 
	\int_{B_{N}} \prod_{i=1}^{2}\|f_{i,l_{i}}\|_{L^2(B_{N^{1/2}}(x))}^{\frac{s}{2}}
	\, dx
	+
	\prod_{i=1}^{2}\|f_{i}\|_{L^2(\mathbb{R}^{6})}^{\frac{s}{4}}\prod_{i=1}^{2}\|f_{i}\|_{L^{s,N}(\mathbb{R}^{6})}^{\frac{s}{4}}.
	\end{split}
	\end{displaymath}
	The second term on the right hand side is harmless.
	By using Plancherel's theorem and a finitely overlapping property, the first term on the right hand side is bounded by
	\begin{gather*}
	\lesssim \max_{l_{1},l_{2}} \biggl[\biggl(
	\int_{B_{N}} \prod_{i=1}^{2}\|f_{i,l_{i}}\|_{L^2(B_{N^{1/2}}(x))}^{\frac{p}{2}}
	\, dx \biggr)	 \prod_{i=1}^{2}
	\sup_{x \in B_{N}}
	\|f_{i,l_{i}} \|_{L^2(B_{N^{1/2}}(x))}^{\frac{s-p}{2}}
	\biggr]
	\\ \lesssim \max_{l_{1},l_{2}} \biggl[\biggl(
	N^{\frac{3(s-p)}{2}}
	\int_{B_{N}} \prod_{i=1}^{2}\|f_{i,l_{i}}\|_{L^2(B_{N^{1/2}}(x))}^{\frac{p}{2}}
	\, dx  \biggr)
	\prod_{i=1}^{2}(\sum_{\theta \in \mathcal{P}_{N^{-1}} }\|f_{i,l_{i}}*\Xi_{\theta}\|_{L^{\infty}(\mathbb{R}^{6})}^2)^{\frac{s-p}{4}}
	\biggr].
	\end{gather*}
	We apply Proposition 8.2 and recover the original function $f_{i}$ from $f_{i,l_{i}}$ by using the inequality $(7.2)$. Then we can bound the above term by
	\begin{displaymath}
	\begin{split}
	&\lesssim N^{\frac{3(s-p)}{2}}
	N^{3(1-\frac{p}{4})}
	\max_{l_{1},l_{2}} \biggl[
	\prod_{i=1}^{2}\|f_{i,l_{i}}\|_{L^{2}(\mathbb{R}^{6})}^{\frac{p}{2}}
	\prod_{i=1}^{2}
	(\sum_{\theta \in \mathcal{P}_{N^{-1}}}\|f_{i,l_{i}}*\Xi_{\theta}\|_{L^{\infty}(\mathbb{R}^{6})}^2)^{\frac{s-p}{4}} \biggr]
	\\& \lesssim N^{\frac{3(s-p)}{2}}
	N^{3(1-\frac{p}{4})} \max_{l_{1},l_{2}} \biggl[
	\prod_{i=1}^{2} \|f_{i}\|_{L^{2}(\mathbb{R}^{6})}^{\frac{s}{4}}
	(\sum_{\theta \in \mathcal{P}_{N^{-1}}}\|f_{i}*\Xi_{\theta}\|_{L^{s}(\mathbb{R}^{6})}^2)^{\frac{s}{8}} \biggr] 
	\\& \lesssim N^{\frac{3(s-p)}{2}}
	N^{3(1-\frac{p}{4})}
	N^{\frac{3}{8}(s-2)}
	\prod_{i=1}^2 \|f_{i}\|_{L^{2}(\mathbb{R}^{6})}^{\frac{s}{4}}
	\|f_{i}\|_{L^{s,N}(\mathbb{R}^{6})}^{\frac{s}{4}}.	
	\end{split}
	\end{displaymath}
	The last inequality follows from H\"{o}lder's inequality. By direct computation, we can see that the exponent of $N$ is $\frac{3s}{4}$. Hence, we obtain the unweighted inequality. Next, by putting $f_{i}\eta_{B_{2N}}$ instead of $f_{i}$ and using a finitely overlapping property, we have
	\begin{gather*}
		\int_{B_{N}} \prod_{i=1}^{2}\|f_{i}\|_{L^2({B_{N^{{1}/{2}}}}(x))}^{\frac{s}{2}}
		\, dx 
		\lesssim_{\epsilon,\nu} N^{\epsilon}
		 N^{\frac{3s}{4}}
		\prod_{i=1}^{2}\|f_{i}\|_{L^2({w}_{B_{N} })}^{\frac{s}{4}}\prod_{i=1}^{2}\|f_{i}\|_{L^{s,N}({w}_{B_{N} })}^{\frac{s}{4}}.
	\end{gather*}
	Observe that
	\begin{displaymath}
	\begin{split}
	w_{B_{N^{1/2}}(x)}(y) \lesssim
	\sum_{c_{B} \in N^{1/2}\mathbb{Z}^{6} }
	1_{B_{N^{1/2}}(x+c_{B})}(y)w_{B_{N^{1/2}}(0)}(c_{B}).
	\end{split}
	\end{displaymath}
	By using this, we have
	\begin{gather*}
		\int_{B_{N}}
		\prod_{i=1}^{2}
		\|f_{i}\|_{L^2(w_{B_{N^{1/2}}(x)})}^{\frac{s}{2}}\,dx \\
		\lesssim
		\sum_{B',B'' \in \mathcal{B}_{N^{\frac{1}{2}}}}
		w_{B_{N^{\frac{1}{2}}}}(c_{B'})w_{B_{N^{\frac{1}{2}}}}(c_{B''})
		\int_{B_{N}}
		\|f_{1}\|^{\frac{s}{2}}_{L^2(B_{N^{\frac{1}{2}}}(x+c_{B'}))}\|f_{2}\|_{L^2(B_{N^{\frac{1}{2}}}(x+c_{B^{''}}))}^{\frac{s}{2}}\,dx
		\\
		\lesssim_{\epsilon, \nu}
		N^{\epsilon}
		 N^{\frac{3s}{4}}
		\prod_{i=1}^{2}
		\|f_{i}\|^{\frac{s}{4}}_{L^2(w_{B_{N}}+a_{i,N})}\|f_{i}\|_{L^{s,N}(w_{B_{N}+a_{i,N}})}^{\frac{s}{4}} 
	\end{gather*}
	for some point $a_{i,N}$.
	This completes the proof of Proposition 8.3.
\end{proof}

Now, we are ready to prove Theorem 8.1. Iterating Proposition 8.3 will lead to the desired inequality.
\subsubsection*{Proof of Theorem 8.1.}
Let $s \geq 6$. Fix a number $r=N^{2^{-M}} \geq 1$. By Bernstein's inequality,
\begin{displaymath}
\begin{split}
\int_{B_{N}}|f_{1}f_{2}|^{\frac{s}{2}} \lesssim \int_{B_{N+r}}\prod_{i=1}^2
\|f_{i}\|_{L^{\infty}(B_{r}(x))}^{\frac{s}{2}}\,dx \lesssim
\int_{B_{N+r}} \prod_{i=1}^{2}\|f_{i}\|_{L^2(w_{B_{r}(x)})}^{\frac{s}{2}}\,dx.
\end{split}
\end{displaymath}
We change the integrand into the average over cubes $B_{r^2}$;
\begin{displaymath}
\begin{split}
\lesssim
r^{-2 \cdot 6}
\int_{B_{N+r+r^2}}\bigl(\int_{B_{r^2}(x)}\prod_{i=1}^{2}\|f_{i}\|_{L^2(w_{B_{r}(y)})}^{\frac{s}{2}}\,dy \bigr)\,dx.
\end{split}
\end{displaymath}
Next, we apply Proposition 8.3 on each cube of side length $r^2$. Then the above term is bounded by
\begin{gather*}
	\lesssim r^{-2 \cdot 6}r^{2 \cdot \frac{3s}{4}}
	\int_{B_{N+r+r^2}} \prod_{i=1}^{2}\|f_{i}\|_{L^2({w}_{B_{r^2}(x+a_{i})})}^{\frac{s}{4}} \prod_{i=1}^{2}\|f_{i}\|_{L^{s,r^2}({w}_{B_{r^2}(x+a_{i})})}^{\frac{s}{4}}
	\,dx
	\\ 
	\lesssim
	r^{-12+\frac{3s}{2}}
	\bigl(\int_{B_{N+r+r^2}}  \prod_{i=1}^{2}\|f_{i}\|_{L^2({w}_{B_{r^2}(x+a_{i})})}^{\frac{s}{2}} \,dx \bigr)^{\frac{1}{2}}
	\prod_{i=1}^{2}
	\bigl(\int_{B_{N+r+r^2}}
	\|f_{i}\|_{L^{s,r^{2}}({w}_{B_{r^2}(x+a_{i})})}^{s} \,dx \bigr)^{\frac{1}{4}}
\end{gather*}
for some $a_{i} \in \mathbb{R}^{6}$. The second inequality follows from H\"{o}lder's inequality. We first consider the $L^{s,r^2}$ norm. By Fubini's theorem,
\begin{displaymath}
\begin{split}
&\int_{B_{N+r+r^2}}
\sum_{\theta \in \mathcal{P}_{r^{-2}}	}\|f_{i}*\Xi_{\theta}\|_{L^{s}({w}_{B_{r^2}(x+a_{i})})}^{s} \,dx
\lesssim r^{12} \|f_{i}\|_{L^{s,r^{2}}(\mathbb{R}^{6})}^{s}
.
\end{split}
\end{displaymath}
By Proposition 6.3, we obtain
\begin{displaymath}
\begin{split}
\|f_{i}\|_{L^{s,r^2}(\mathbb{R}^{6})}
&\lesssim D_{S}(\frac{N}{r^2},s)\|f_{i}\|_{L^{s,N}(\mathbb{R}^{6})} 
\\&\lesssim_{\epsilon} (\frac{N}{r^2})^{(\gamma_{lin}+\epsilon)}\|f_{i}\|_{L^{s,N}(\mathbb{R}^{6})}
=N^{(\gamma_{lin}+\epsilon)(1-\frac{2}{2^M})}\|f_{i}\|_{L^{s,N}(\mathbb{R}^{6})}.
\end{split}
\end{displaymath}
By using these two inequalities, we get
\begin{gather*}
	\int_{B_{N}}|f_{1}f_{2}|^{\frac{s}{2}} 
	\lesssim_{\epsilon}  r^{-6+\frac{3s}{2}}
	N^{{s(\gamma_{lin}+\epsilon)}(\frac{1}{2}-\frac{1}{2^M})} 
	\bigl(\int_{B_{N+r+r^2}}  \prod_{i=1}^{2}\|f_{i}\|_{L^2({w}_{B_{r^2}(x+a_{i})})}^{\frac{s}{2}} \,dx \bigr)^{\frac{1}{2}}
	\prod_{i=1}^{2}
	\|f_{i}\|_{L^{s,N}(\mathbb{R}^{6})}^{\frac{s}{4}}.
\end{gather*}
Repeating this process again on the first term gives
\begin{gather*}
	\int_{B_{N+r+r^2}}  \prod_{i=1}^{2}\|f_{i}\|_{L^2({w}^2_{B_{r^2}(x+a_{i})})}^{\frac{s}{2}} \,dx \\ \lesssim_{\epsilon} 
	r^{2\cdot (-6+{\frac{3s}{2}})}
	N^{{s(\gamma_{lin}+\epsilon)}(\frac{1}{2}-\frac{2}{2^M})} \prod_{i=1}^{2}
	\biggl(\int_{B_{N+r+r^2+r^4}}  \prod_{i=1}^{2}\|f_{i}\|_{L^2({w}^2_{B_{r^4}(x+b_{i})})}^{\frac{s}{2}} \,dx \biggr)^{\frac{1}{2}}
	\|f_{i}\|_{L^{s,N}(\mathbb{R}^{6})}^{\frac{s}{4}}
\end{gather*}
for some $b_{i} \in \mathbb{R}^{6}$. Combining these two inequalities leads to
\begin{gather*}
	\int_{B_{N}}|f_{1}f_{2}|^{\frac{s}{2}} \\
	\leq  C_{s,\epsilon} 
	r^{2\cdot (-6+ {\frac{3s}{2}})}N^{{s(\gamma_{lin}+\epsilon)}(\frac{3}{4}-\frac{2}{2^M})}
	\biggl(\int_{B_{N+r+r^2+r^4}}  \prod_{i=1}^{2}\|f_{i}\|_{L^2({w}_{B_{r^4}(x+b_{i})})}^{\frac{s}{2}} \,dx \biggr)^{\frac{1}{4}}
	\prod_{i=1}^{2}
	\|f_{i}\|_{L^{s,N}(\mathbb{R}^{6})}^{\frac{s}{4}(1+\frac{1}{2})}.
\end{gather*}
By repeating this process $M-2$ times more, recalling that $r=N^{2^{-M}}$ and using H\"{o}lder's inequality, we obtain
\begin{gather*}
	\int_{B_{N}}|f_{1}f_{2}|^{\frac{s}{2}} \\
	\leq C_{s,\epsilon}^{M}
	r^{C'}r^{M(-6+ {\frac{3s}{2}})}
	N^{{s(\gamma_{lin}+\epsilon)}(1-\frac{M}{2^M})} 
	\prod_{i=1}^{2}\|f_{i}\|_{L^2(\mathbb{R}^{6})}^{\frac{s}{2^{M+1}}}
	\prod_{i=1}^{2}
	\|f_{i}\|_{L^{s,N}(\mathbb{R}^{6})}^{\frac{s}{2}(1-\frac{1}{2^M})}
	\\  \leq N^{\frac{M \log C_{s,\epsilon} }{\log N}}
	N^{\frac{C''}{2^M}}N^{\frac{(-6+ {\frac{3s}{2}})M}{2^{M}}}N^{{s(\gamma_{lin}+\epsilon)}(1-\frac{M}{2^M})}
	\prod_{i=1}^{2}
	\|f_{i}\|_{L^{s,N}(\mathbb{R}^{6})}^{\frac{s}{2}}
	.
\end{gather*}
By using a standard localization argument and summing over cubes $B_{N}$ and raising to the power $\frac{1}{s}$, we have
\begin{gather*}
	\||f_{1}f_{2}|^{\frac{1}{2}}\|_{L^s(\mathbb{R}^{6})}
	\leq 
	N^{\frac{M \log C_{s,\epsilon} }{s\log N}+ \frac{C'}{2^M}+\frac{(-6+ {\frac{3s}{2}})M}{2^{M}s}+(\gamma_{lin}+\epsilon)(1-\frac{M}{2^M})} \prod_{i=1}^{2} \|f_{i}\|_{L^{s,N}(\mathbb{R}^{6})}^{\frac{1}{2}}.
\end{gather*}
By the definition of $\gamma_{bil}$, Proposition 6.1 and 6.2, we have
\begin{displaymath}
\begin{split}
N^{\gamma_{bil}(s)} \lesssim_{\epsilon} N^{\frac{M \log C_{s,\epsilon} }{s\log N}+\frac{C'}{2^M}+\frac{(-6+ {\frac{3s}{2}})M}{2^{M}s}+(\gamma_{lin}(s)+\epsilon)(1-\frac{M}{2^M})},
\end{split}
\end{displaymath}
by using $\gamma_{lin} \leq \gamma_{bil}$ and rearranging this inequality,
\begin{displaymath}
\begin{split}
N^{\gamma_{bil}(s)\frac{M}{2^M}} \lesssim_{\epsilon} N^{\frac{M \log C_{s,\epsilon} }{s\log N}+\frac{C'}{2^M}+\frac{M}{2^M} \cdot \frac{3}{2}(1-\frac{4}{s})+\epsilon}.
\end{split}
\end{displaymath}
By taking $N$ and $M$ sufficiently large and $\epsilon$ sufficiently small, we obtain
\begin{displaymath}
\begin{split}
\gamma_{bil}(s) \leq \frac{3}{2} -\frac{6}{s}.
\end{split}
\end{displaymath}
This completes the proof.
\qed

\subsection*{Acknowledgements}
The author thanks the referee for some comments that improved the presentation of our results.

				\end{document}